\documentclass[12pt, a4paper]{article}

\usepackage{dsfont} 

\usepackage{graphicx}
\usepackage[dvipsnames]{xcolor}
\usepackage{subfigure}

\usepackage{csquotes}
\usepackage[backend=bibtex,
		bibencoding=utf8
	    sorting=nyt,
	    style=numeric,
	    ]{biblatex} 

\usepackage{amssymb} 
\usepackage{amsfonts}
\usepackage{amsmath}
\usepackage{amsthm}
\usepackage{enumitem}

\newcommand{\RR}{\mathbb{R}} 
\newcommand{\ZZ}{\mathbb{Z}} 
\newcommand{\CC}{\mathbb{C}} 
\newcommand{\EE}{\mathbb{E}} 

\newcommand{\LL}{\mathcal{L}}
\newcommand{\NN}{\mathcal{N}}

\newcommand{\XX}{\mathbf{X}} 
\newcommand{\YY}{\mathbf{Y}} 
\newcommand{\ZZZ}{\mathbf{Z}} 

\newcommand{\de}{\mathrm{d}} 
\newcommand{\e}[1]{\mathrm{e}^{#1}} 
\newcommand{\ii}{\mathrm{i}} 
\newcommand{\tr}{\mathrm{tr}} 
\newcommand{\supp}{\mathrm{supp}} 

\newcommand{\abs}[1]{\left| #1\right|}
\newcommand{\norm}[1]{\left\| #1\right\|}
\newcommand{\scal}[2]{\left\langle#1,#2\right\rangle}
\newcommand{\floor}[1]{\left\lfloor#1\right\rfloor}

\newcommand{\dprod}{\displaystyle\prod}
\newcommand{\dint}{\displaystyle\int}

\newcommand{\I}[2]{\left[#1{,}#2\right]}

\newcommand{\DPP}{\mathrm{DPP}}

\newtheorem{theorem}{Theorem}[section]
\newtheorem{proposition}{Proposition}[section]
\newtheorem{lemma}{Lemma}[section]

\newenvironment{changemargin}[2]{\begin{list}{}{%
\setlength{\topsep}{0pt}%
\setlength{\leftmargin}{0pt}%
\setlength{\rightmargin}{0pt}%
\setlength{\listparindent}{\parindent}%
\setlength{\itemindent}{\parindent}%
\setlength{\parsep}{0pt plus 1pt}%
\addtolength{\leftmargin}{#1}%
\addtolength{\rightmargin}{#2}%
}\item }{\end{list}}

\newcommand{\rev}[1]{\textcolor{black}{#1}}

\usepackage{hyperref}
\hypersetup{colorlinks=true,linkcolor=blue,urlcolor=blue}

\usepackage{hyperref}
\hypersetup{colorlinks=true,linkcolor=blue,urlcolor=blue, citecolor=blue}

\begin{document}

%
%
%
%

\title{Projections of determinantal point processes}
\author{Adrien Mazoyer, Jean-François Coeurjolly and Pierre-Olivier Amblard}

\maketitle

\begin{abstract}
	Let $\mathbf x=\{x^{(1)},\dots,x^{(n)}\}$ be a space filling-design of $n$ points defined in $\I{0}{1}^d$. In computer experiments, an important property seeked for $\mathbf x$ is a nice coverage of $\I{0}{1}^d$. This property could be desirable  as well as for any  projection of $\mathbf x$ onto  $\I{0}{1}^\iota$ for $\iota<d$ . Thus we expect that $\mathbf x_I=\{x_I^{(1)},\dots,x_I^{(n)}\}$, which represents the design $\mathbf x$ with coordinates associated to any index set $I\subseteq\{1,\dots,d\}$, remains regular in $\I{0}{1}^\iota$ where $\iota$ is the cardinality of $I$. This paper examines the conservation of nice coverage by projection using spatial point processes,  and more specifically using the class of determinantal point processes. We provide necessary conditions on the kernel defining these processes, ensuring that  the projected point process $\XX_I$ is repulsive, in the sense that its pair correlation function is uniformly bounded by~1, for all $I\subseteq\{1,\dots,d\}$. We present a few examples, compare them using a new normalized version of  Ripley's function. Finally, we illustrate the interest of this research for Monte-Carlo integration.
\end{abstract}

%

\section*{Introduction}
	
Space-filling designs, e.g. Latin hypercubes~\cite{McKayetal79,Owen92}, low discrepancy sequences~\cite[e.g.][]{Halton64,Sobol67}, are popular methods in computer experiments. These computational methods are becoming unavoidable to simulate complex phenomena \cite[e.g.][Chapter 5]{Santneretal13}.
 A space-filling design corresponds to a set $\mathbf x=\{x^{(1)},\dots,x^{(n)}\}$ of $n$ points generated in a bounded domain,  for instance $\I{0}{1}^d$ in the following. Usually, the dimension $d$ represents the number of factors (or covariates) on which the numerical code depends. The $k$th coordinates of points from $\mathbf x$  then represent the values of the $k$th factor.  
Intuitively, points issued from a space filling-design tend to regularly cover the domain $\I{0}{1}^d$. The quality of this coverage  can be a priori evaluated by standard criteria such as maximin distance or $L^2$ discrepancy \cite[see e.g.][]{Owen13}.

Frequently in computer  experiments, some factors are a posteriori found to be inactive \cite[see][and references therein]{Sunetal19}. 
 If the experiment is to be performed again, an inactive factor must be discarded to avoid numerical errors and to decrease complexity. 
But if $k$ factors are discarded, the experimental space-filling design should be done again, this time on $\I{0}{1}^{d-k}$. This induces a new complexity and is expensive. A cheaper strategy is to keep the first space-filling design, but use its projection onto $\I{0}{1}^{d-k}$ by discarding the adequate factors (or coordinates). 
 However,  the projected points should  provide a good coverage of  $\I{0}{1}^{d-k}$.  Therefore, an additional property of the initial space-filling design should be the conservation of the ``nice coverage" property for any subsets of  the coordinates.

This additional property has already been considered in the literature for low discrepancy type designs \cite[see e.g.][]{Sunetal19}.
Our work in contrast considers spatial point processes as experimental designs. 
For the question we address, we set $\mathbf x_I=\{x_I^{(1)},\dots,x_I^{(n)}\}$ to be the design obtained from the design $\mathbf x$ by keeping the factors (coordinates) indexed by the index set $I\subseteq \{1,\dots,d\}$.
For example,  when $I=\{1,\dots,d-1\}$, $\mathbf x_I$ corresponds to the set $\mathbf x$ where the $d$th coordinate of each point is discarded. We let $\XX$ to be the spatial point process generating $\mathbf x$ and $\XX_I$ the process generating $\mathbf x_I$. In the spatial statistics literature, the pair correlation function (denoted by $g$) is the most standard way for characterizing the pairwise dependence between points, see e.g. \cite{MollerWaagepeterson04}: $g_\XX(x,y)$ measures the probability to observe a pair of distinct points at $(x,y)$, normalized by the same probability under the Poisson case, i.e. under the situation where there is no interaction between points (see Section~\ref{sec:backgr} for a more formal definition of the pair correlation function). 
A point process for which $g_\XX<1$, i.e. $g_\XX(x,y)<1$ for all $x,y \in \I{0}{1}^d$ is qualified as a repulsive point process in \cite{MollerWaagepeterson04,Illianetal08}.
Thanks to repulsiveness, points of a repulsive point process tend to cover more regularly the space than a Poisson point process does.
For the application to computer experiments which motivates our study, we  intend to develop point process models which are repulsive in all directions, i.e. spatial point processes $\XX$ such that $g_{\XX_I}<1$ for all $I\subseteq \{1,\dots,d\}$.

Several classes of spatial point processes are able to generate regular patterns. Among them,  Mat\'ern hard-core processes \cite{Teichmannetal13}, Gibbs point processes \cite{MollerWaagepeterson04,Dereudre19} or determinantal point processes \cite{Lavancieretal15_1} are appealing for many applications. In particular Gibbs point processes have been considered to generate space-filling designs in~\cite{Francoetal08, Dupuyetal15}. The authors build a specific Gibbs model by parameterizing its Papangelou conditional intensity as a Strauss hard-core model with constraints on the marginals. The resulting patterns  look regular and the points cover regularly the space. However, Gibbs point processes have the drawback of not having their moments available in a closed form. In particular, the intensity as well as the pair correlation function \cite[see e.g.][]{Dereudre19} are not available analytically. Even worse, using Monte-Carlo simulations in \cite{Illianetal08} show that the pair correlation function of Strauss hard-core models are not uniformly bounded by~1 . Although Mat\'ern hard-core processes are more tractable, their  pair correlation function suffer from the same problem \cite[see][]{Teichmannetal13}.

Determinantal point processes (DPPs for short) have been introduced in \cite{Macchi75} as ``fermion'' processes to model the position of particles that repel each other. This class of processes is known for very appealing properties, in particular for its tractability: explicit expressions for the intensity functions are available. Therefore, a growing attention has been paid to DPPs from a theoretical point of view \cite[e.g.][]{Soshnikov00, ShiraiTakahashi03,Houghetal09,Decreusefondetal16}, and more recently in the statistics community \cite{Lavancieretal15_1,BardenetHardy20}. In particular, one of the main characteristics of a DPP is that, by construction, its pair correlation function is uniformly bounded by~1.
DPPs  are defined through a kernel $K:B\times B \to \mathbb C$ which characterizes the distribution of $\XX$ and thus which characterizes also its moments.
The main result of this paper 
concerns necessary conditions (expressed by Assumption~\eqref{hyp:sep_kern}) on the form of  kernel $K$ 
to ensure that the projected pattern $\XX_I$ remains repulsive, i.e. such that $g_{\XX_I}<1$.

The  paper is organized as follows. Section~\ref{sec:backgr} contains a brief background on spatial point processes and in particular on DPPs. Section~\ref{sec:DPPproj} deals with 
the statistical description of the projected point process $\XX_I$. In particular we provide a closed form for the pair correlation $g_{\XX_I}$ when $\XX$ is a DPP defined on $\I{0}{1}^d$ with kernel $K$ satisfying a separability assumption.
Examples of models satisfying this 
separability condition are presented and discussed in Section~\ref{sec:ex}. They are compared using an original summary statistic, defined as a normalized version of  Ripley's function \cite[see e.g.][]{MollerWaagepeterson04} based on the sup norm.
We illustrate in Section~\ref{sec:app} the interest of the models developed in this research. To mimic situations which occur in computer experiments, we consider the Monte-Carlo integration for  $\int_{\I{0}{1}^\iota} f_I(u)\mathrm d u$, for any function $f_I:\I{0}{1}^\iota \to \mathbb R$ and any $I\subseteq \{1,\dots,d\}$ with cardinality $\iota=1,\dots,d$. 
We demonstrate that the single initial design defined on $\I{0}{1}^d$ and its projections can be used to achieve this task efficiently. Proofs of our results are postponed to appendices.


\section{Background and notation} \label{sec:backgr}

\subsection{Spatial point processes}

A spatial point process~$\XX$ defined on a Borel set~$B\subseteq\RR^d$ is a locally finite measure on $B$, (for measure theoretical details, see e.g. \cite{MollerWaagepeterson04} and references therein) whose realization is of the form 
$\{x^{(1)}, \ldots , x^{(k)}\}\in B^k$
where~$k$ is the realization of a random variable and the~$x^{(i)}$'s represent the events. We assume that $\XX$ is simple meaning that two events cannot occur at the same location. Thus, $\XX$ is viewed as a locally finite random set.

In most cases, the distribution of a point process $\XX$ can be described by its intensity functions {$\rho_\XX^{(k)}:B^k\rightarrow \RR^+$, $k\in\mathbb{N} \setminus \{0\}$. By Campbell Theorem \cite[e.g.][]{MollerWaagepeterson04}, $\rho_\XX^{(k)}$ is characterized by the following integral representation: for any non-negative measurable function $h:B^k \to \RR^+$  
\begin{align}
\EE\Bigg[&\sum_{x^{(1)}, \ldots , x^{(k)} \in \XX}^{\neq} h\left(x^{(1)}, \ldots , x^{(k)}\right)\Bigg] \nonumber\\
&	 = \int_{B^k} \rho_\XX^{(k)}\left(x^{(1)},\ldots,x^{(k)}\right) h\left(x^{(1)}, \ldots , x^{(k)}\right) \de x^{(1)}\ldots\de x^{(k)}\label{eq:rhok_def}		
\end{align}  
where~$\neq$ over the summation means that~$x^{(1)}, \ldots , x^{(k)}~$ are pairwise distinct points. Intuitively, for any pairwise distinct points $x^{(1)},\ldots,x^{(k)}\in B$, \linebreak$\rho_\XX^{(k)}\left(x^{(1)},\ldots,x^{(k)}\right)\de x^{(1)}\ldots\de x^{(k)}$ is the probability that $\XX$ has a point in each~of the $k$ infinitesimally small sets around~$x^{(1)},\ldots,x^{(k)}$ with volumes $\de x^{(1)},\ldots,\de x^{(k)}$, respectively. When $k=1$, this yields the intensity function and we simply denote it by $\rho_{\XX}=\rho_{\XX}^{(1)}$.
The second order intensity~$\rho_\XX^{(2)}$ is used to define the pair correlation function
	\begin{equation} \label{eq:pcf_def}
		g_\XX(x^{(1)},x^{(2)}) = \frac{\rho_\XX^{(2)}(x^{(1)},x^{(2)})}{\rho_\XX(x^{(1)})\rho_\XX(x^{(2)})}\,
 	\end{equation}
	for pairwise distinct $x^{(1)}$, $x^{(2)} \in B$,  and
	where $g_\XX(x^{(1)},x^{(2)})$ is set to 0 if~$\rho_\XX(x^{(1)})$ or~$\rho_\XX(x^{(2)})$ is zero.
	 By convention, $\rho_\XX^{(k)}\left(x^{(1)},\ldots,x^{(k)}\right)$ is set to 0 if~$x^{(i)}=x^{(j)}$ for some~$i\neq j$. 
	Therefore~$g_\XX(x,x)$ is also set to 0 for all~$x\in B$. The pair correlation function (pcf for short) can be used to determine the local interaction between points of~$\XX$~located at $x$ and $y$:~$g_\XX(x,y) > 1$ characterizes positive correlation between the points;~$g_\XX(x,y)=1$ means there is no interaction (typically a Poisson point process);~$g_\XX(x,y) < 1$ characterizes negative correlations. A point pattern is often referred to as a repulsive point process, if $g(x,y)<1$ for any $x,y\in B$ \cite[e.g.][Sec. 6.5]{Illianetal08}.

A point process $\XX$ with constant intensity function on $B$ is said to be homogeneous. A pcf with constant intensity is said to be invariant by translation (resp. isotropic) if $\rho_\XX^{(2)}(x^{(1)},x^{(2)})$ depends only on $x^{(2)}-x^{(1)}$ (resp. on $\|x^{(2)}-x^{(1)}\|$ for a norm to be defined).


	\subsection{Determinantal point processes (DPPs)}

In this section, the class of continuous DPPs is introduced. We restrict our attention to DPPs defined on a compact set $B\subset \RR^d$. A point process~$\XX$ on~$B$ is said to be a DPP on $B$ with kernel $K:B\times B \to \mathbb C$ if for any $k\ge 1$ its~$k$th order intensity function is given by
\begin{equation} \label{eq:rhok_DPP}
\rho_\XX^{(k)}\left(x^{(1)},\ldots,x^{(k)}\right) = \det\left[K\left(x^{(i)},x^{(j)}\right)\right]_{i,j=1}^k
\end{equation}
and we simply denote by $\XX\sim \DPP_B(K)$. 
We assume in this work  that $K$ is a continuous covariance function and refer the interested reader to more general situations to~\cite{Houghetal09}.
The intensity of $\XX$ is given by $\rho_\XX(x)=K(x,x)$ and its pcf by
\begin{equation} \label{eq:pcf_DPP}
g_\XX(x,y) = 1-\frac{\abs{K(x,y)}^2}{K(x,x)K(y,y)}.
\end{equation}
The popularity of DPPs relies mainly upon \eqref{eq:rhok_DPP}-\eqref{eq:pcf_DPP}: all moments of $\XX$ are explicit and by assumption on $K$,
$g_\XX(x,y)<1$ for any $x,y\in B$.
From~\eqref{eq:pcf_DPP} and the continuity of~$K$, it is worth mentioning that $g_\XX$ is continuous on the diagonal, i.e. $g_\XX(x,y) \to 0$ when $y\to x$ for any $x\in B$.

From Mercer's Theorem \cite[Sec. 98]{RieszSznagy90},  kernel $K$ admits the following decomposition for any $x,y\in B$
\begin{equation} \label{eq:Mercer}
K(x,y) = \sum_{j\in \NN} \lambda_j \phi_j(x)\overline{\phi_j(y)}
\end{equation}
where $\NN$ is a countable set (e.g. $\mathbb{N},\; \ZZ,\;\ZZ^d,\ldots$), $\{\phi_j\}_{j\in \NN}$ are eigenfunctions associated to $K$ and form an orthonormal basis of the space of square-integrable functions~$L^2(B)$. 
$\{\lambda_j\}_{j\in \NN}$ are the eigenvalues of $K$.
Let us mention that we abuse notation when referring $\phi_j$'s and $\lambda_j$'s to as eigenfunctions and eigenvalues of $K$. These should require to introduce the notion of integral operator with kernel $K$ \cite[e.g.][]{DebnathMikusinski05} acting on $L^2(B)$. To simplify the reading, we make the misnomer to consider  kernel $K$ instead of the associated integral operator.
We define the trace of kernel $K$ on $B$ by
\[
\tr_B(K) = \int_B K(x,x)\de x = \sum_{j\in\NN} \lambda_j.
\]
In the following, the kernels we consider are assumed to have finite trace, and are called trace class kernels.
The existence of a DPP with kernel $K$ is ensured if $K$ is trace class, and is  such that $\lambda_j \le 1$ for any $j\in \NN$ \cite[e.g.][Theorem 4.5.5]{Houghetal09}.

A kernel such that its non-zero eigenvalues are equal to 1 is called a ``projection kernel''. In particular, if $\XX$ is a ``projection DPP'', i.e. $\XX\sim\DPP_B(K)$ where $K$ is a ``projection kernel'', then the number of points of $\XX$ in $B$, is almost surely constant and equal to the trace of $K$. Notice that the name ``projection kernel'' is not related at all with the projection transformation we are studying here. This terminology seems commonly used though \cite[e.g.][]{Houghetal06,Houghetal09, McCullaghMoller06,Lavancieretal15_1}.

The homogeneous case is often considered later. A DPP $\XX$ with kernel $K$ is said to be homogeneous, if $K$ is the restriction on $B\times B$ of a kernel $\bar K$ defined on $\RR^d\times \RR^d$ which is stationary, i.e. satisfies
\[
\bar K(x,y) = \bar K(0,x-y),\quad x,y\in \RR^d.
\]
In that case, we abuse notation, identify $K$ with $\bar K$ and refer to $K$ as a stationary kernel.
It is worth pointing out that if $K$ admits a Mercer's decomposition with respect to the Fourier basis
\begin{equation} \label{eq:fourier_basis}
	\phi_j(x) = \e{2\ii\pi \scal{j}{x}}
\end{equation}
 where $\scal{\cdot}{\cdot}$ denotes the inner product on $\RR^d$,  then $K$ is stationary.

\section{Projection of a spatial point process and applications to DPPs \label{sec:DPPproj}}
	\subsection{Projection of a spatial point process}

In this work, we consider projection of spatial point processes, i.e. keeping a given number of coordinates from the original spatial point process. Such a framework requires that the original point process $\XX$ is defined on a compact set $B \subset \RR^d$: otherwise, the configuration of  points of the projected point processes may not form locally finite configuration, as also noticed in the two-dimensional case in~\cite[p. 17]{Baddeley07}. 

This section presents a few notation and characterization of projected spatial point processes. Let $I$ be a subset of $\overline{d} = \{1,\ldots, d\}$ with cardinality $\abs{I}=\iota$. In the following, we let~$B\subset\RR^d$ be a compact set, which can be written as $B_1\times \dots \times B_d$. We denote by~$B_I$ the set
	$
	B_I = \prod_{i\in I} B_i\,
	$
	with $B=B_{\overline{d}}$
	and by~$P_I$ the orthogonal projection of~$\RR^d$ onto~$\RR^{\iota}$.
	For any point process~$\XX$ defined on such a compact~$B\subset\RR^d$, the projected point process~$\XX_I=P_I\XX$ is then defined on~$B_I$. For any $x \in B$, we often use the notation $x_I$ to denote $P_I x$. We sometimes use the notation $\XX_{\overline d}=\XX$ when $I=\overline d$. The following Lemma provides a general way to evaluate intensity functions of $\XX_I$. 
\begin{lemma} \label{lem:XI}
	 Let $I\subset\overline{d}$ and let $\XX$ be a spatial point process defined on a compact set $B\subset \RR^d$. Then, 
		 for any $k\ge 1$ such that $\rho_\XX^{(k)}$ exists,  $\rho_{\XX_I}^{(k)}$ is well-defined and
			\begin{align} \label{eq:rhok_proj}
				\rho_{\XX_I}^{(k)}&\left(x^{(1)},\ldots,x^{(k)}\right)  \nonumber\\
				& = \int_{(B_{I^c})^k} \rho_{\XX}^{(k)} \left(\left(x^{(1)},u^{(1)}\right),\ldots,\left(x^{(k)},u^{(k)}\right)\right) \de u^{(1)}\ldots\de u^{(k)} 
			\end{align}
	for any pairwise distinct $x^{(1)},\ldots,x^{(k)}\in B_I$
where~$I^c=\overline{d}\setminus I$.
\end{lemma}
Lemma~\ref{lem:XI} is obtained by a simple application of Campbell's Theorem. Its proof is provided in~\ref{app:proof_XI} for the sake of completeness.
We now turn to the core of this paper which is the study of projected determinantal point processes.


\subsection{Distribution of $\XX_I$ when $\XX \sim\DPP_B(K)$}
	
According to \eqref{eq:rhok_proj}, the $k$th order intensity function of the projected point process $\XX_I$ is given by
\begin{align}
\rho_{\XX_I}^{(k)}\left(x^{(1)},\ldots,x^{(k)}\right) & = \int_{({B_{I^c}})^k} \rho_{\XX}^{(k)} \left((x,u)^{(1)},\ldots,(x,u)^{(k)}\right) \de u^{(1)}\ldots\de u^{(k)} \nonumber\\[2ex]
& = \int_{({B_{I^c}})^k} \det\left[K((x,u)^{(i)},(x,u)^{(j)})\right]_{i,j=1}^k \de u^{(1)}\ldots\de u^{(k)} \nonumber\\[2ex]
& = \sum_{\sigma\in S_k} (-1)^{k-C(\sigma)} \int_{(B_{I^c})^k} \prod_{i=1}^k K\left((x,u)^{(i)},(x,u)^{(\sigma(i))}\right) \nonumber\\
& \phantom{\sum_{\sigma\in S_k} (-1)^{k-C(\sigma)} \int_{(B_{I^c})^k}}\de u^{(1)}\ldots\de u^{(k)} \label{eq:rhok_proj_DPP1}
\end{align}		
where~$S_k$ is the symmetric group on~$\overline{k}=\{1,\dots,k\}$, 
$C(\sigma)$ is the number of disjoint cycles of $\sigma$, and $(x,y)^{(i)}$ denotes $(x^{(i)},y^{(i)})$. Without any assumption on  kernel~$K$, there is no chance to reduce~\eqref{eq:rhok_proj_DPP1} any further, i.e. to have an explicit form for $\rho_{\XX_I}^{(k)}$ and thus $g_{\XX_I}$. Therefore, without additional assumption, it is difficult to assess whether $g_{\XX_I}$ is smaller than 1 or not. The following assumptions will allow us to solve this problem.

\begin{description}[style=unboxed,leftmargin=0cm]
\item[Assumption \eqref{hyp:sep_kern}] For $I \subseteq \overline d$, the  kernel $K$ can be written as 
\begin{equation}\label{hyp:sep_kern} \tag{\(H[I]\)}
	K(x,y) = K_I(x_I, y_I)  K_{I^c}(x_{I^c}, y_{I^c})
\end{equation}
where \mbox{$K_I:B_I\times B_I\rightarrow\CC$} and \mbox{$K_{I^c}:B_{I^c}\times B_{I^c}\rightarrow\CC$} are two continuous covariance functions.

Assumption~\eqref{hyp:sep_kern} implies that $K$ admits the Mercer's decomposition \linebreak
\mbox{$K(x,y)  = \sum_{j \in \NN} \lambda_j \phi_j(x) \overline{\phi}_j(y)$,} where $\NN=\NN_I \times \NN_{I^c}$, $\lambda_j = \lambda^{(I)}_{j_I}\lambda^{(I^c)}_{j_{I^c}}$, \linebreak $\phi_j(x)=\phi^{(I)}_{j_I}(x_I)\phi^{(I^c)}_{j_{I^c}}(x_{I^c})$ for $j=(j_I,j_{I^c})$, $x=(x_I,x_{I^c})$. Here, for $\bullet=I,I^c$, $\{\phi_{j_\bullet}^{(\bullet)}\}_{j_\bullet \in \NN_\bullet}$ 
is a set of normalized eigenfunctions  of $K_\bullet$, (and thus an orthonormal basis of $L^2(B_\bullet)$) and $\lambda_{j_\bullet}$ denote the eigenvalues of $K_\bullet$.

If $K$ admits a Mercer's decomposition with respect to the Fourier basis such that its eigenvalues satisfy the above separability property, then~\eqref{hyp:sep_kern} is satisfied. Hence, Fourier basis appears as a natural basis and leads us to consider, the following natural extension of~\eqref{hyp:sep_kern} that would be assumed for any $I \subseteq \overline d$.
\item[Assumption \eqref{hyp:sep_id_kern}] We assume that the kernel $K$ satisfies \eqref{hyp:sep_kern} for any $I \subseteq \overline d$, 
is stationary and can be written as the product of $d$ one-dimensional stationary kernels:
\begin{equation} \label{hyp:sep_id_kern} \tag{\(H^{\prime}\)}
	K(x-y) = \prod_{i=1}^d K_i(x_i-y_i),\quad x,y\in B
\end{equation}
where for any $i\in\overline{d}$, each $K_i:B_i\times B_i\rightarrow\CC$ is a stationary continuous kernel. 
\item[Assumption \eqref{hyp:sep_id_kern0}] We will also focus on the particular case where all kernels are identical, i.e. $K_i\equiv K_0$ for all $i\in\overline{d}$:
\begin{equation} \label{hyp:sep_id_kern0} \tag{\(H^{\prime\prime}\)}
	K(x-y) = \prod_{i=1}^d K_0(x_i-y_i),\qquad x,y\in B.
\end{equation}
Assumption~\eqref{hyp:sep_id_kern0} is well-suited to the situation where we have no a priori information on the projection $P_I$ from the initial point process $\XX$ we want to study.

\end{description}
We could remove the stationarity assumption in Assumption~\eqref{hyp:sep_id_kern}. However, as revealed by Sections~\ref{sec:ex} and~\ref{sec:app}, stationarity allows us to plot pcfs or Ripley's functions of $\XX_I$ for any $I$. It thus provides a visual interpretation of regularity properties for $\XX_I$.  Furthermore, going back to one motivation of this paper, there is a priori no reason to construct a design which favours particular spatial areas. Thus, considering a stationary kernel which ensures that the intensity is constant makes sense.


%
\begin{theorem} \label{theo:proj_DPP_sep_kern}
	Let $I\subseteq\overline{d}$  and $\XX\sim \DPP_B(K)$ such that $K$ satisfies \eqref{hyp:sep_kern}.
Then the $k$th order intensity function of the projected point process $\XX_I$ is given by
\begin{align} \label{eq:rhok_proj_DPP_sep_kern_res}
\rho_{\XX_I}^{(k)}\left(x^{(1)},\ldots,x^{(k)}\right)  
= \sum_{\sigma\in S_k}&(-1)^{k-C(\sigma)}\left[\prod_{i=1}^k K_I\left(x^{(i)},x^{(\sigma(i))}\right)\right]\\
& \tr_{B_{I^c}}(K_{I^c})^{k-c(\sigma)}\prod_{\varepsilon\in\mathcal{S}(\sigma)}\tr_{B_{I^c}}\left(K_{I^c}^{(c(\varepsilon))}\right)\nonumber
\end{align}		
where $c(\sigma)$ is the size of the support $\supp(\sigma) = \left\{i\in\overline{k}\text{  s.t.  }\sigma(i)\neq i\right\}$, $\mathcal{S}(\sigma)$ is the set of disjoint cycles of $\sigma$ with non-empty support, $C(\sigma)$ is the number of disjoint cycles of $\sigma$ (including those with empty support) and for a kernel $K$, $K^{(m)}$ for $m>1$, stands for the iterated kernel defined by \linebreak \mbox{$K^{(m)}(x,y)=\int K^{{(m-1)}}(x,z)K(z,y) \mathrm d z$} (with $K^{(1)}=K$).
In particular, the intensity of $\XX_I$ is given by $\rho_{\XX_I}(x) = K_I(x, x) \tr_{B_{I^c}}\left(K_{I^c}\right)$ and its pcf is given by

\begin{equation} \label{eq:pcf_proj_DPP}
		g_{\XX_I}(x,y) = 1-\frac{\tr_{B_{I^c}}\left(K_{I^c}^{(2)}\right)}{\tr_{B_{I^c}}(K_{I^c})^2}(1-g_{\YY^{(I)}}(x,y))
\end{equation}
for any pairwise distinct $x,y\in B_I$ and where $\YY^{(I)}\sim\DPP_{B_I}(K_I)$.
\end{theorem}

We focus in Theorem~\ref{theo:proj_DPP_sep_kern} on intensity functions. However, we can prove a full characterization of the distribution of $\XX_I$ via its Laplace functional. This is detailed in~\ref{sec:laplace}. In particular, Theorem~\ref{thm:laplace} shows that $\XX_I$ is distributed as an infinite superposition of independent DPPs, each with kernel $\lambda_{l}^{(I^c)}K_I$. In particular, if $K_{I^c}$ is a projection kernel, $\XX_I$ is a finite superposition of $M=\tr_{B_{I^c}}\left(K_{I^c}\right)$ i.i.d. DPPs with kernel $K_I$. Such finite superposition corresponds \cite[see][]{Houghetal06,Decreusefondetal16} to the distribution of an $\alpha$-DPP on $B_I$ with kernel $-\alpha^{-1} K_I$ where $\alpha=-M^{-1}$. $\alpha$-DPPs (and $\alpha$-determinants) are introduced in~\cite{ShiraiTakahashi03}. And it can indeed be checked from~\eqref{eq:rhok_proj_DPP_sep_kern_res} that when $K_{I^c}$ is a projection kernel, the intensity functions of $\XX_I$ are given by
\begin{align*}
		\rho_{\XX_I}^{(k)}\left(x^{(1)},\ldots,x^{(k)}\right)   & = \sum_{\sigma\in S_k} \alpha^{k-C(\sigma)}\left[\prod_{i=1}^k-\alpha^{-1} K_I\left(x^{(i)},x^{(\sigma(i))}\right)\right] \\
	& := \mathrm{det}_{\alpha} \left[-\alpha^{-1} K_I\left(x^{(i)},x^{(j)}\right)\right]_{i,j=1}^k
	\end{align*}

Equation~\eqref{eq:pcf_proj_DPP} is in our opinion the most interesting result of this paper. It reveals the repulsiveness nature of $\XX_I$. Let us examine this in details. Since $\YY^{(I)}$ is a DPP  with kernel $K_I$, it satisfies $0\le g_{\YY^{(I)}}\le 1$, which allows us to rewrite~\eqref{eq:pcf_proj_DPP} as
\begin{equation}\label{eq:pcf_proj_DPP2}
	0 \le 1 - g_{\XX_I}(x,y) = \frac{\tr_{B_{I^c}}\left(K_{I^c}^{(2)}\right)}{\tr_{B_{I^c}}(K_{I^c})^2}(1-g_{\YY^{(I)}}(x,y)) \le 
	\frac{\tr_{B_{I^c}}\left(K_{I^c}^{(2)}\right)}{\tr_{B_{I^c}}\left(K_{I^c}\right)^2}.
\end{equation}
The lower-bound of~\eqref{eq:pcf_proj_DPP2} means that $g_{\XX_I}\le 1$, i.e. $\XX_I$ is indeed a repulsive point process on $B_I$. Furthermore, the upper-bound measures in some sense the loss of repulsion and more precisely, how $g_{\XX_I}$ gets closer to 1 which corresponds the pcf of a Poisson point process. To be more precise, let us focus on the particular case~\eqref{hyp:sep_id_kern}. We have in this situation
\[
  g_{\XX_I}(x,y) \ge  1 - \prod_{i\in I^c} \frac{\tr_{B_i}\left(K_i^{(2)}\right)}{\tr_{B_i}\left(K_i\right)^2}.
\]
For each $i\in\overline{d}$, $\tr_{B_i}\left(K_i^{(2)}\right)/\tr_{B_i}\left(K_i\right)^2<1$. Therefore, when $|I^c|=d-\iota$ is large, $1-g_{\XX_I}$ is bounded by a product of large number of quantities smaller than 1, and thus the pcf of $\XX_I$ gets closer and closer to the pcf of a Poisson point process. It is even more obvious when $K$ satisfies~\eqref{hyp:sep_id_kern0}. In that case, for any $x,y\in B_I$
\[
g_{\XX_I}(x,y) \ge 1- {\kappa_0}^{d-\iota} \quad \text{ where} \quad 
\kappa_0 = \frac{\tr_{B_0}\left(K_0^{(2)}\right)}{\tr_{B_0}\left(K_0\right)^2}.
\]
For example, when $\iota=d-1$, i.e. when one skips only one coordinate: $g_{\XX_I}(x,y) \ge 1-\kappa_0>0$ and this constant is reached when $y\to x$. Since, $g_\XX(x,y)\to 0$ when $y\to x$, one can clearly measure the loss of repulsion as soon as one skips one coordinate.



 \section{Examples} \label{sec:ex}

In this section, we present particular examples of kernels defined on \mbox{$B=\I{0}{1}^d$} and satisfying~\eqref{hyp:sep_id_kern}, thus ensuring that $\XX_I$ is repulsive for any $I\subseteq \overline d$. Then, these examples are compared for different sets $I$ through their pair correlation function or through a normalized Ripley's function.

All our kernels examples have Mercer's decomposition  defined with respect to the Fourier basis \eqref{eq:fourier_basis}, the natural basis which allows~\eqref{hyp:sep_id_kern} to be satisfied.


\subsection{Gaussian kernel} \label{ssec:Gauss}
 			
The Gaussian kernel (see e.g. \cite{Lavancieretal15_1})
\[
K(x,y) = \rho\exp\left(-\norm{\frac{x-y}{\alpha}}\right)
\]
where $\norm{\cdot}$ denotes the Euclidean norm, is the typical example satisfying~\eqref{hyp:sep_id_kern0}, where $K_0$ is defined for any $x,y\in \I{0}{1}$ by:
\[
 K_0(x-y) = \rho^{1/d}\exp\left(-\left(\frac{x-y}{\alpha}\right)^2\right).
 \]
The existence of $\XX\sim\DPP_B(K)$ is ensured if $\alpha$ is such that $\rho(\alpha\sqrt{\pi})^d\le 1$.
For any $I\subseteq \overline d$, the pcf of $\XX_I$  is derived from Theorem~\ref{theo:proj_DPP_sep_kern}: for any pairwise distinct $x,y\in B_I$
\begin{equation}\label{eq:pcf_proj_gauss}
	g_{\XX_I}(x,y) = 1-\kappa_2^{d-\iota} \exp\left(-2\norm{\frac{x-y}{\alpha}}^2\right)
\end{equation}
with
\begin{equation} \label{eq:kappa2}
\kappa_2 = \frac{\tr_{B_0}\left(K_0^{(2)}\right)}{\tr_{B_0}\left(K_0\right)^2}\approx \frac{\sum_{j\in\ZZ}\exp\left(-2(j\alpha\pi)^2\right)}{\left(\sum_{j\in\ZZ}\exp\left(-(j\alpha\pi)^2\right)\right)^2}.
\end{equation}
This approximation comes from  the Fourier approximation of  kernel $K$ detailed in \cite[Section~4]{Lavancieretal15_1}. Note that for all $I\subseteq \overline d$ and $x,y \in B_I$, we use with a slight abuse the same notation $\|x-y\|$ for the Euclidean norm in $\RR^{\iota}$. 

This class of examples is of particular interest due to the isotropy property of $g_{\XX_I}$. The pcfs $g_{\XX_I}$ for different sets $I$ can be represented on the same plot. For $d=10,10^2,10^3,10^4$, Figure~\ref{fig:pcf_proj_gauss} represents the pcfs of a Gaussian DPP $\XX$ (solid lines) and its successive projections. The intensity parameter and $\alpha$ are set to $\rho_\XX=500$ and $\alpha^{-1} = \rho_\XX^{1/d}\sqrt{\pi}$. Note that the abscissa corresponds to $\|x-y\|$ for $x,y\in B_I$ for different sets $I$. Thus the differences should be understood carefully. 
Figure~\ref{fig:pcf_proj_gauss} confirms that the pcf of $\XX_I$ is upper-bounded by~1, lower-bounded by $1-\kappa_2^{d-\iota}$ and gets closer to 1 when $\iota$ decreases.

\begin{figure} 
\centering
\begin{changemargin}{-1.5cm}{0cm}
\begin{tabular}{c c}
$d=10$, $\rho=500$ & $d=100$, $\rho=500$ \\
\includegraphics[scale=0.3]{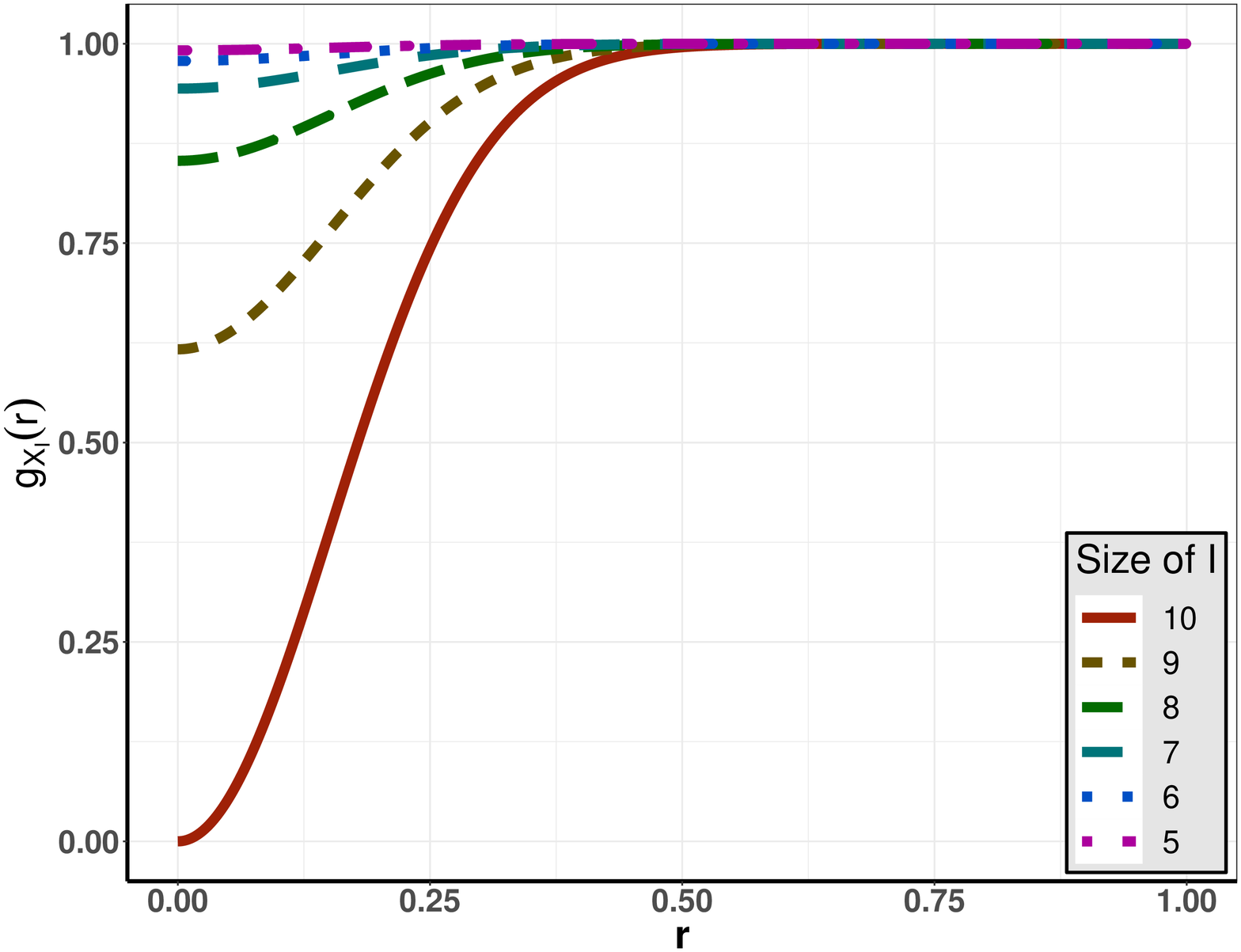} & \includegraphics[scale=0.3]{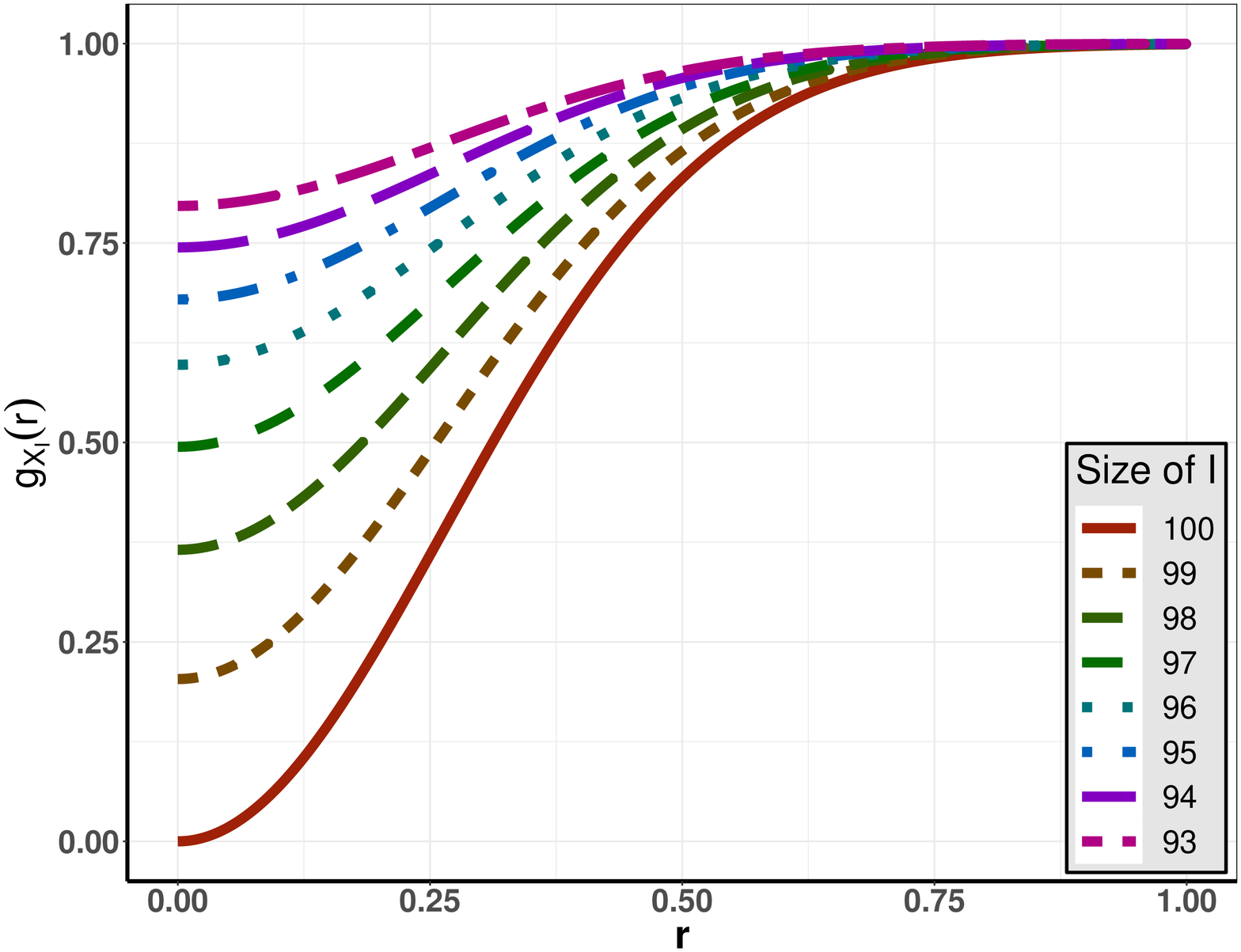} \\
$d=10^3$, $\rho=500$ & $d=10^4$, $\rho=500$ \\
\includegraphics[scale=0.3]{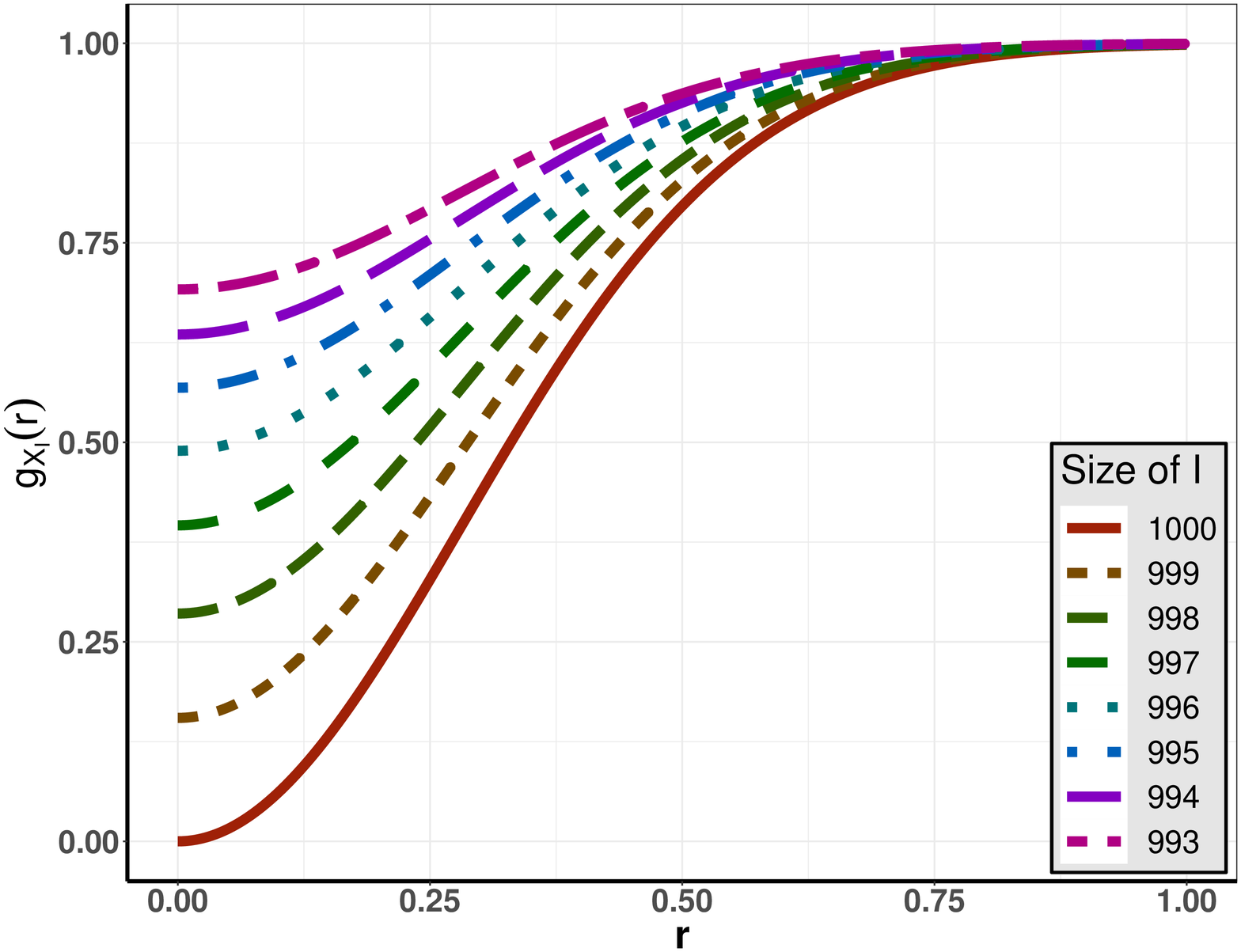} & \includegraphics[scale=0.3]{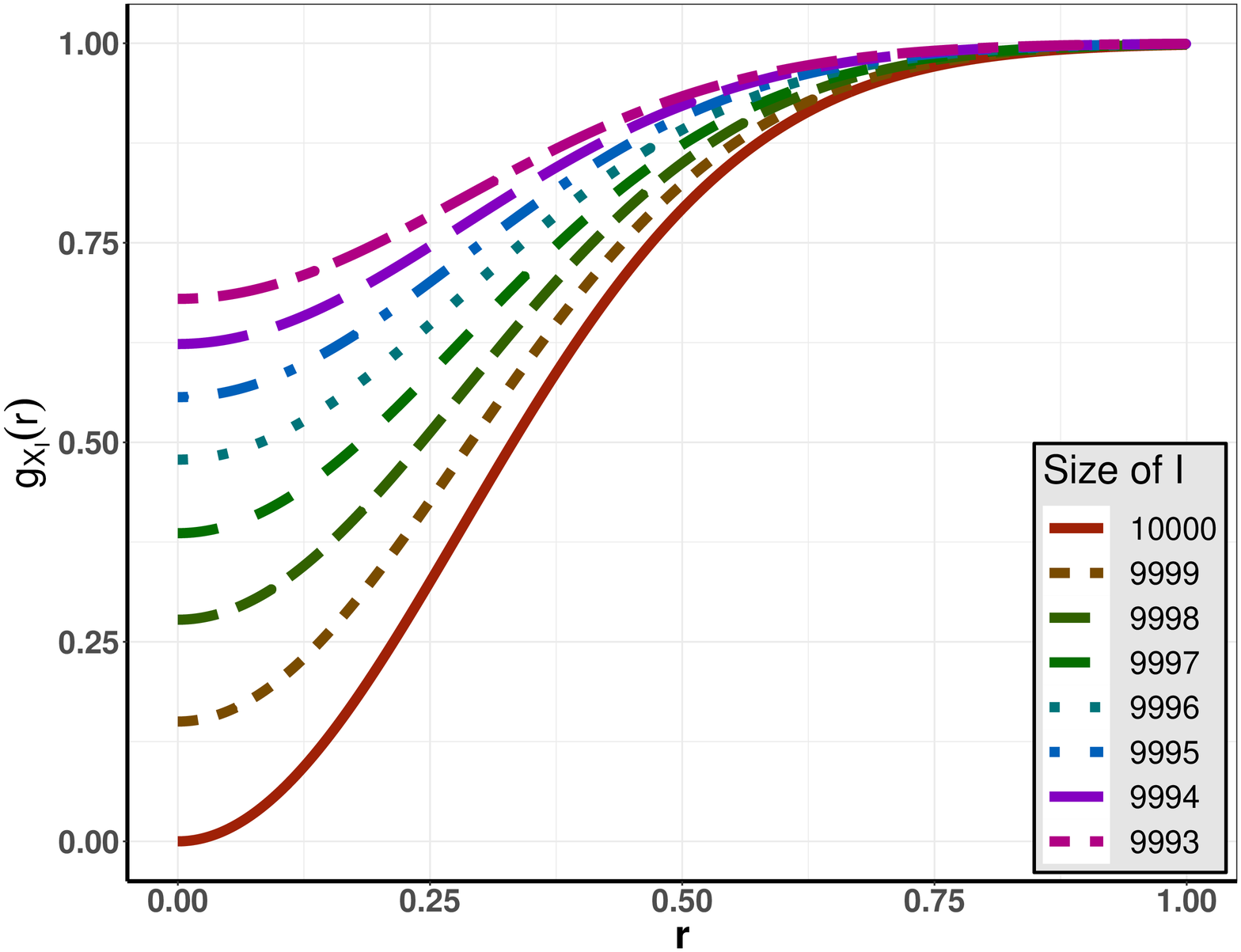} \\
\end{tabular}
\end{changemargin}
\caption{
Pair correlation functions of the Gaussian DPP $\XX$ (solid lines) with intensity $\rho_\XX = 500$ and $\alpha^{-1} =\rho_{\XX}^{1/d}\sqrt{\pi}$ and its successive projections $\XX_I$ ($\abs{I}=d-1,\ldots$; dotted and dashed lines) for $d=10$ (top left), $10^2$ (top right), $10^3$ (bottom left) and $10^4$ (bottom right).}
\label{fig:pcf_proj_gauss}
\end{figure}	
		
\subsection{$L^1$-Exponential kernel}

We now consider an exponential kernel, defined with respect to the \mbox{$L^1$-norm} instead of the Euclidean norm:
\begin{equation} \label{eq:L1Gauss_kernel}
K(x-y) = \rho\exp\left(-\norm{\frac{x-y}{\alpha}}_1\right).
\end{equation}			
The kernel~\eqref{eq:L1Gauss_kernel}  is referred to as the $L^1$-Exponential kernel in the following. It also constitutes a natural example as it satisfies~\eqref{hyp:sep_id_kern0} where $K_0$ is defined for any $x,y\in \I{0}{1}$ by:
\[
K_0(x-y) = \rho^{1/d}\exp\left(-\abs{\frac{x-y}{\alpha}}\right).
\]
The existence of $\XX\sim\DPP_B(K)$ is ensured if $\alpha$  is such that $\rho(2\alpha)^d\le1$. According to Theorem~\ref{theo:proj_DPP_sep_kern}, for any $I\subseteq \overline d$, the pcf of $\XX_I$ is given for any pairwise distinct $x,y\in B_I$ by
\begin{equation}\label{eq:pcf_proj_L1gauss}
g_{\XX_I}(x,y) = 1-\kappa_1^{d-\iota} \exp\left(-2\norm{\frac{x-y}{\alpha}}_1\right)
\end{equation}
with
\begin{equation} \label{eq:kappa1}
	\kappa_1 = \frac{\tr_{B_0}\left(K_0^{(2)}\right)}{\tr_{B_0}\left(K_0\right)^2} \approx \frac{\sum_{j\in \ZZ} \left(1+(2\pi\alpha j)^2\right)^{-2}}{\left(\sum_{j\in \ZZ} \left(1+(2\pi\alpha j)^2\right)^{-1}\right)^2}
\end{equation}
where the approximation corresponds again to the Fourier approximation. For $d=10,10^2,10^3,10^4$, Figure~\ref{fig:pcf_proj_L1gauss} represents the pcfs of an $L^1$-Exponential DPP $\XX$ (solid lines) and its successive projections with respect to the $L^1$-norm. The intensity parameter and  $\alpha$ are set to $\rho_\XX=500$ and $\alpha^{-1}=2\rho_\XX^{1/d}$. The conclusion drawn from Figure~\ref{fig:pcf_proj_L1gauss} is similar to the one from Figure~\ref{fig:pcf_proj_gauss}: the pcf of $\XX_I$ is upper-bounded by~1, lower-bounded by $1-\kappa_1^{d-\iota}$ and tends to 1 when $\iota$ decreases. We could be tempted to compare Figures~\ref{fig:pcf_proj_gauss} and~\ref{fig:pcf_proj_L1gauss} and conclude that the Gaussian DPP seems more repulsive. However, remember that both models are not isotropic with respect to the same norm. We provide in Section~\ref{sec:comparison} a summary statistic which allows us to correctly compare these models.

\begin{figure} 
\centering
\begin{changemargin}{-1.5cm}{0cm}
\begin{tabular}{c c}
$d=10$, $\rho=500$ & $d=100$, $\rho=500$ \\
\includegraphics[scale=0.3]{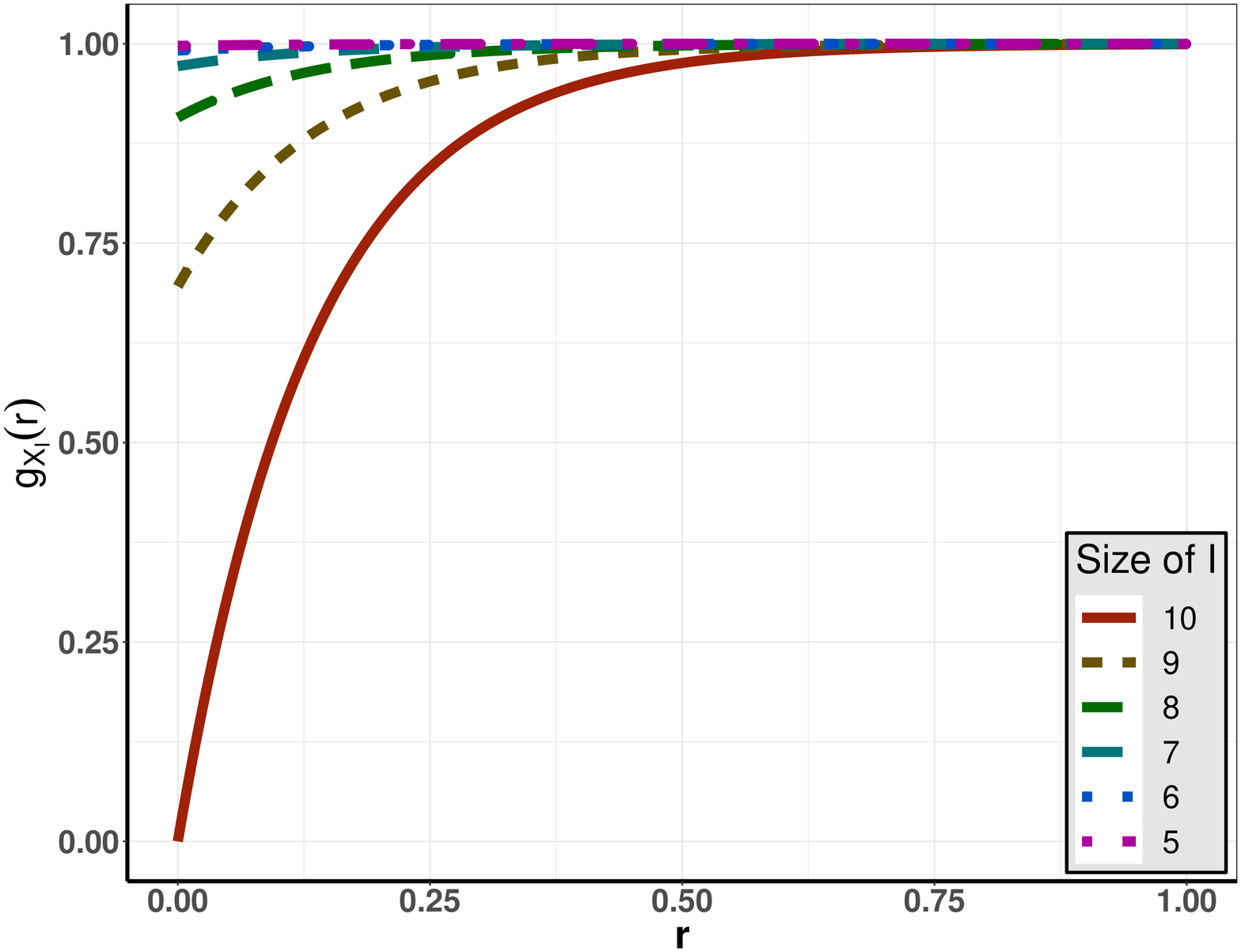} & \includegraphics[scale=0.3]{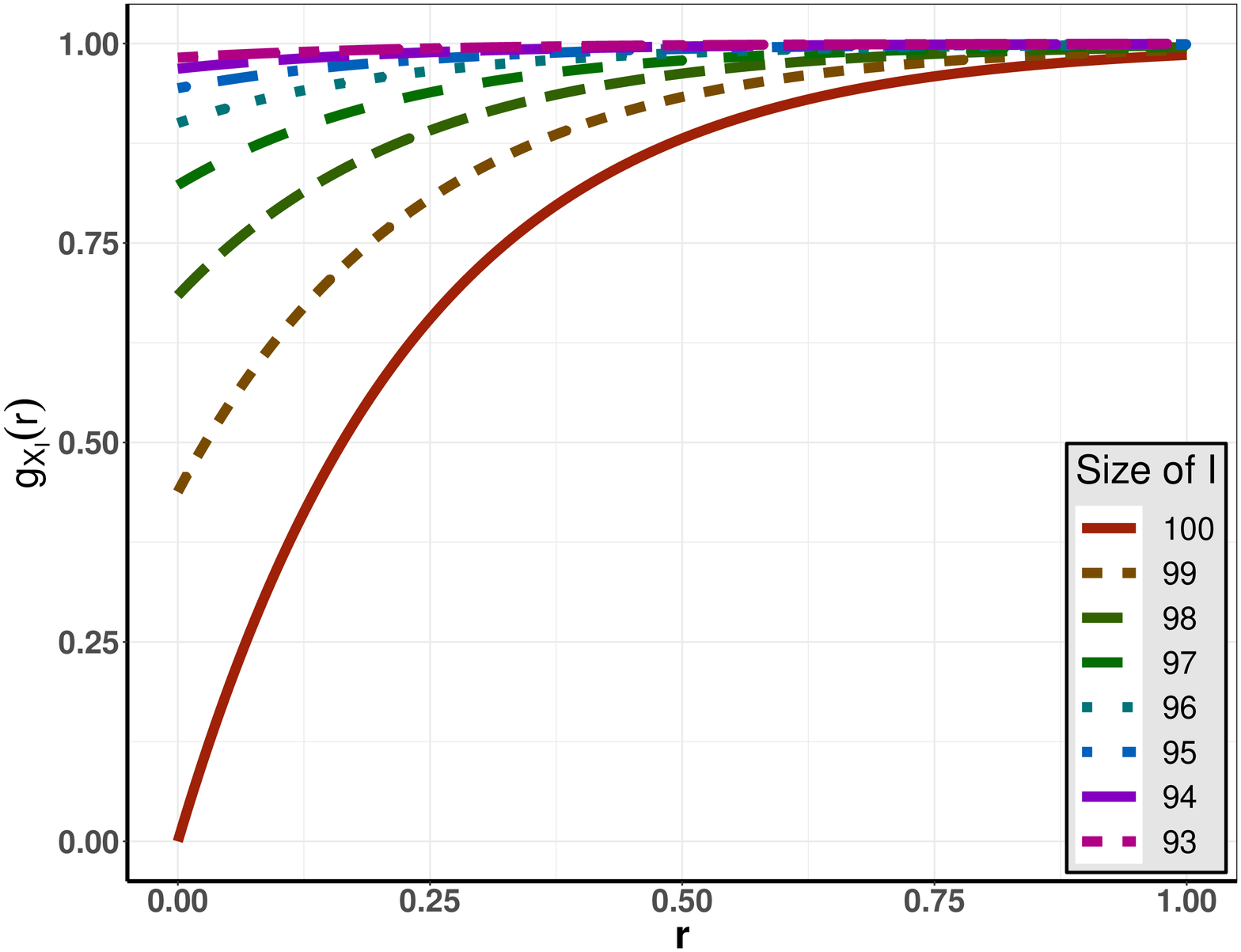} \\
$d=10^3$, $\rho=500$ & $d=10^4$, $\rho=500$ \\
\includegraphics[scale=0.3]{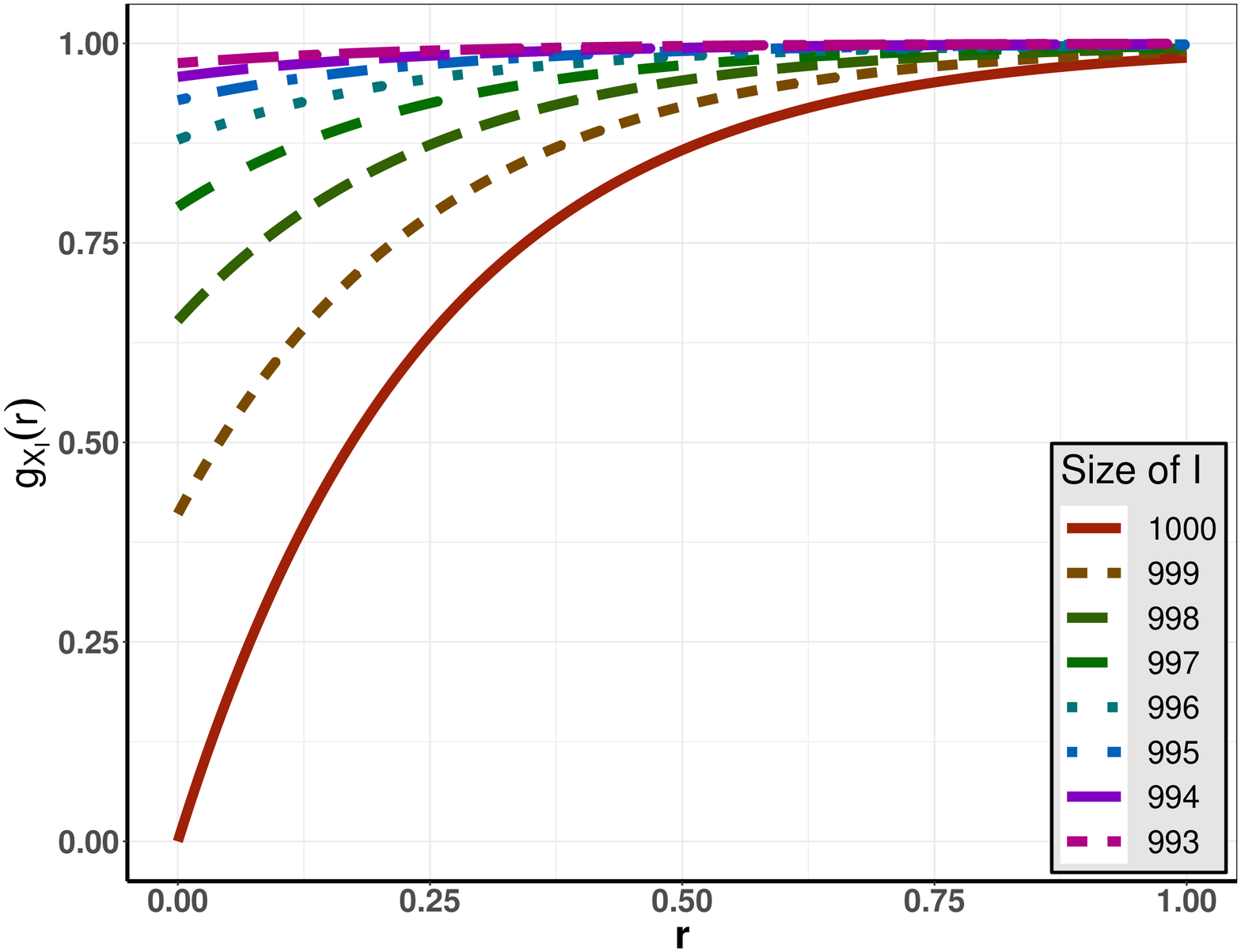} & \includegraphics[scale=0.3]{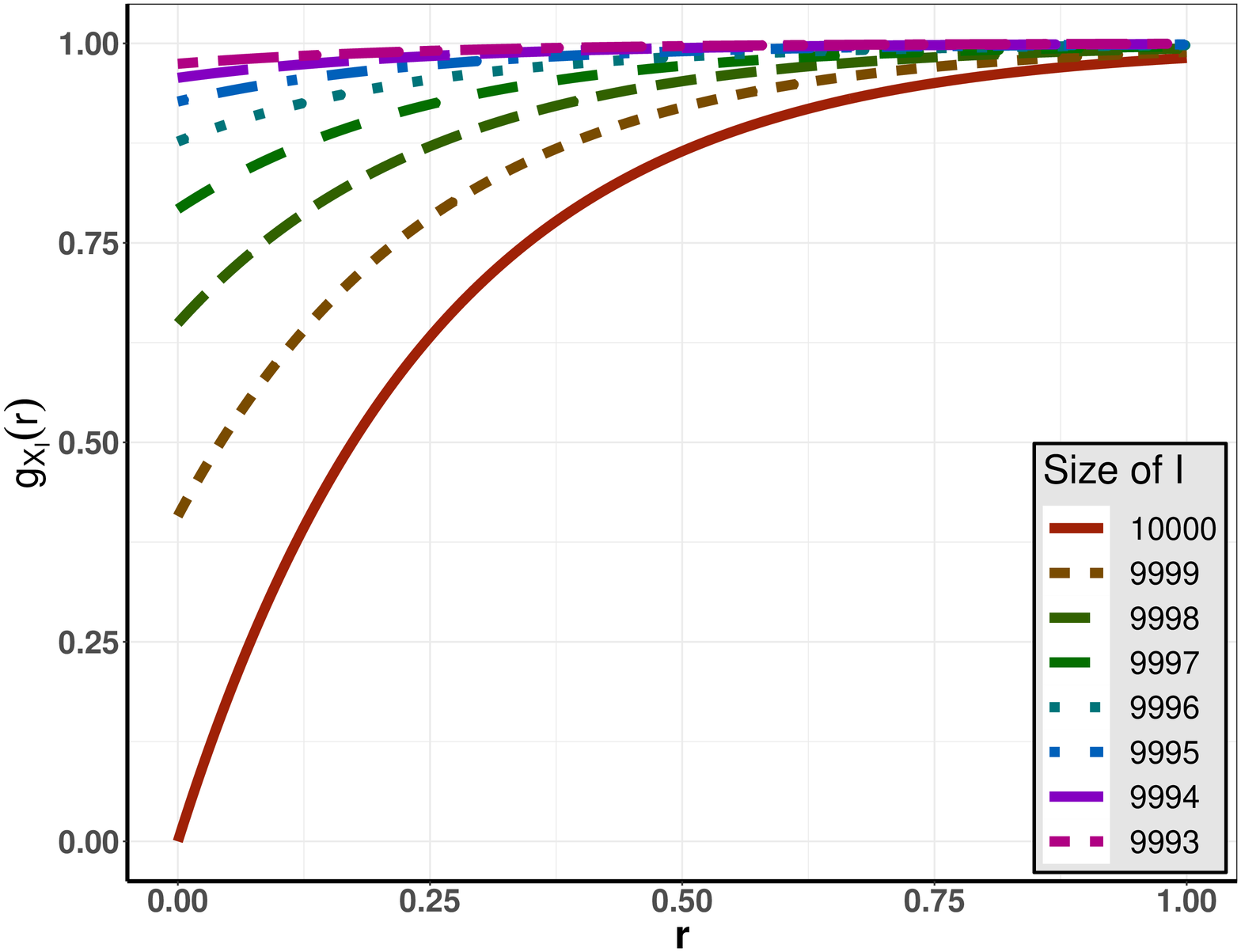}
\end{tabular}
\end{changemargin}
\caption{
Pair correlation functions of the $L^1$-Exponential DPP $\XX$ (solid lines) with intensity $\rho_\XX = 500$ and  $\alpha^{-1} =2\rho_{\XX}^{1/d}$ and its successive projections $\XX_I$ ($\abs{I}=d-1,\ldots$; dotted and dashed lines) for $d=10$ (top left), $10^2$ (top right), $10^3$ (bottom left) and $10^4$ (bottom right).
}
		\label{fig:pcf_proj_L1gauss}
\end{figure}	

\subsection{Dirichlet kernels} \label{sec:dirichlet}
	The two examples considered so far satisfy~\eqref{hyp:sep_id_kern0} by definition. The next one is a projection kernel which only satisfies~\eqref{hyp:sep_id_kern}. \rev{For $\iota=1,\dots,d$, we let $\{\phi_j^{(\iota)}\}_{j\in\ZZ^\iota}$ denote the $\iota$-dimensional Fourier basis.} We  consider $d$ positive integers $(n_i)_{i\in\overline{d}}$ and for $i\in\overline{d}$ the following  one-dimensional stationary kernel:
\[
			{K_i(x-y) = \sum_{j\in E_i} \phi^{(1)}_j(x-y)}, \rev{\quad x,y\in \I{0}{1},}
\]
where $E_i=\{a_i, a_i+1,\dots, n_i-1+a_i\}$ is a set of $n_i$ consecutive integers and $a_i\in\ZZ$. Then, we construct a kernel $K$ as
\[
	K(x-y) = \prod_{i=1}^d K_i(x_i-y_i) = \sum_{j\in E_N} \phi^{(d)}_j(x-y), \rev{\quad x,y\in \I{0}{1}^d}
\]
where $E_N = \prod_i E_i$. It is worth pointing out that
 kernel $K$ can be written as
\begin{equation}
	K(x-y)=\prod_{i=1}^d\left(\sum_{j=a_i}^{n_i-1+a_i} \phi_{j}^{(1)}(x_i-y_i)\right) = 
\phi_{a}^{(d)}(y-x)\prod_{i=1}^d\left(\sum_{j=0}^{n_i-1} \phi_{j}^{(1)}(x_i-y_i)\right)
\end{equation}
where $a = (a_i)_{i\in\overline{d}}$.
Therefore, according to Remark (4) \rev{from~\cite[p. 48]{Houghetal09}}, the choice of the $E_i$'s does not influence the distribution of the  DPP with kernel $K$. 
Remark that, if the $n_i$'s are all odd numbers and if we choose $a_i=-\lfloor n_i/2\rfloor$, the kernel $K$ equals 
\begin{equation}\label{eq:pureDirichlet}
	K(x-y) = \prod_{i=1}^d D_{\floor{\frac{n_i}{2}}}(x_i-y_i)
\end{equation}
where $D_p$ is the Dirichlet kernel \cite[e.g.][]{Zygmund02} with parameter $p$. That terminology justifies the name Dirichlet kernel for this model.
In the general case, and unambiguously we set $a_i=0$ for any $i$ and thus
consider $E_N = \{j\in\mathbb{N}^d\,:\,j_i< n_i,\, i=1\dots d\}$
\begin{equation} \label{eq:prod_nDirichlet_kernel}
	K(x-y) = \sum_{j\in E_N} 
	\e{2 \ii \pi \scal{j}{x-y}}.
\end{equation}
A DPP on $B$ with kernel given by~\eqref{eq:prod_nDirichlet_kernel} is referred to as an $(N,d)$-Dirichlet kernel. 
From Theorem~\ref{theo:proj_DPP_sep_kern}, for any $I\subseteq \overline d$, the pcf of $\XX_I$ is given for any $x,y\in B_I$ by
\begin{align} 
g_{\XX_I}(x,y) &= 1-\frac{1}{N} \sum_{j\in F_{N_I}} \left[\prod_{i\in I}\left(1-\frac{\abs{j_i}}{n_i}\right)\right] \phi^{(\iota)}_j(x-y) \nonumber\\
& =
1-\frac{1}{N}  \prod_{i\in I}     \sum_{|j|<n_i}\left(1-\frac{\abs{j}}{n_i}\right)
 \phi^{(1)}_j(x_i-y_i)
 \label{eq:pcf_proj_prod_N_Dirichlet}
\end{align}
where $F_{N_I} = \{j\in\ZZ^{\iota}\,:\,\abs{j_i}<n_i,\, i\in I\}$.
The pcf $g_{\XX_I}$ is bounded from below by $1-\prod_{i\in I^c}n_i^{-1}$. 

\rev{A question remains on the factorization of $N=\prod_{i=1}^d n_i$. We consider the factorization which minimizes the fluctuation of the $n_i$'s. For instance, when $N=100$ and $d=6$, we use $N=5\times 5 \times 2 \times 2 \times 1 \times 1$ while for $N=800$ we use the decomposition $N=5\times 5 \times 4 \times 2\times2 \times 2$.}
 
The next section provides a summary statistics well-suited to the comparison of the three examples we have so far considered.

\subsection{Normalized Ripley's function}\label{sec:comparison}

Since the Gaussian DPP and $L^1$-Exponential DPP are isotropic but with respect to a different norm and since the $(N,d)$-Dirichlet DPP is even not isotropic, it is hard to compare these different examples. 
In addition to the pcf, a way of characterizing regularity or repulsion in the literature is obtained by analyzing   Ripley's  function \cite[e.g.][]{MollerWaagepeterson04}. This function is not adapted for our framework. However, since all models satisfy~\eqref{hyp:sep_id_kern}, we propose to compare them through a normalized version of \rev{Ripley}'s function based on the sup norm $\|\cdot\|_\infty$.

For  a stationary spatial point process  $\XX$  on $B \subseteq \RR^d$, we define the normalized \mbox{$d$-dimensional} Ripley's function for some $r\ge 0$ by
\begin{equation}\label{eq:Ripley}
R_{\XX}(r) = \frac{\EE \left( N_\XX(B_{d,\infty}(0,r)\setminus 0) \mid 0 \in \XX \right)  }{\EE \left( N_{\mathbf\Pi}(B_{d,\infty}(0,r) \setminus 0) \mid 0 \in \mathbf\Pi \right)}	
\end{equation}	
where $B_{d,\infty}(0,r) = \{w\in \RR^d: |w_i|\le r, i=1,\dots,d\}$ is the $d$-dimensional ball with norm $\|\cdot\|_\infty$ centered at zero with radius $r$, $\mathbf\Pi$ is a homogeneous Poisson point process on $B$ with intensity $\rho$ and $N_\XX(A)$ (resp. $N_{\mathbf\Pi}(A)$) denotes the number of points of $\XX$ (resp. $\mathbf{\Pi}$) in a bounded subset $A\subset \RR^d$.
Assuming that $\XX$ has a pcf, say $g_\XX$, it is known from the properties of the second factorial moment that
\begin{equation} \label{eq:Ripley2}
	R_{\XX}(r) = \frac{\int_{B_{d,\infty}(0,r)} g_\XX(w) \mathrm d w}{\int_{B_{d,\infty}(0,r)} g_{\mathbf\Pi}(w) \mathrm d w} = (2r)^{-d} \int_{B_{d,\infty}(0,r)} g_\XX(w) \mathrm d w.
\end{equation}
Obviously, under the Poisson case $R_\XX=1$ whereas $R_\XX<1$ means that $\XX$ is repulsive. More precisely, the more $R_\XX<1$ the more repulsive $\XX$. 
We now present the interest of $R_\XX$ in our context.

\begin{proposition} \label{prop:Ripley}
Let $\XX \sim \DPP_{B}(K)$ be a DPP with kernel $K$ satisfying~\eqref{hyp:sep_id_kern}. Then, for any $I \subseteq \overline{d}$
\begin{equation} \label{eq:Ripley_proj}
R_{\XX_I} (r) = 1-\left(\prod_{i\in I^c}\frac{\tr_{B_i}\left(K_i^{(2)}\right)}{\tr_{B_i}\left(K_i\right)^2}\right)\left( \prod_{i\in I} \int_{0}^1 \frac{\abs{K_i(tr)}^2}{K_i(0)^2} \mathrm d t\right)
\end{equation}
In particular, if $K$ satisfies \eqref{hyp:sep_id_kern0}:
\begin{equation} \label{eq:Ripley_proj0}
	R_{\XX_I} (r) = 1 -  \kappa_0 ^{d-\iota} \left( \int_{0}^1 \frac{\abs{K_0(tr)^2}}{K_0(0)^2} \de t\right)^{\iota}
\end{equation}
where 
\[
\kappa_0 = \frac{\tr_{B_0}\left(K_0^{(2)}\right)}{\tr_{B_0}\left(K_0\right)^2}.
\]
\end{proposition}
The proof of this result follows directly from~\eqref{eq:pcf_proj_DPP} and~\eqref{eq:Ripley2}. 
Focusing on examples presented in the previous sections, we have
\begin{changemargin}{-.5cm}{0cm}
\[
	R_{\XX_I}(r) = \begin{cases}
		1 - \kappa_2^{d-\iota} \left( \dint_0^1 \e{-2t^2r^2/\alpha^2} \de t\right)^{\iota} & \text{for a Gaussian DPP,}\\[2ex]
		1 - \kappa_1^{d-\iota}\left( \dint_0^1 \e{-2tr/\alpha} \mathrm\de t\right)^{\iota} &\text{for an $L^1$-Exponential DPP,}\\[2ex]
	1-\dfrac{1}{N} \dprod_{i\in I} \sum_{\abs{j}< n_i}\left(1-\frac{\abs{j}}{n_i}\right) \mathrm{sinc}(2\pi jr) & \text{for an $(N,d)$-Dirichlet DPP}
	\end{cases}
\]
\end{changemargin}
where $\kappa_2$ and $\kappa_1$
are defined by~\eqref{eq:kappa2} and \eqref{eq:kappa1}, respectively and $\mathrm{sinc}$ is the cardinal sine function.

Figures~\ref{fig:Ripley_6d}-\ref{fig:Ripley_100d} investigate the situation for $d=6,10,100$ respectively. Ripley's functions for point processes $\XX_I$ based on the three models exposed in this section are depicted. The intensity is set to  $\rho_\XX=500$ and $\iota=d-i$ for $i=0,\dots,5$. 
The Gaussian DPP and $L^1$-Exponential DPP satisfy~\eqref{hyp:sep_id_kern0}, and so we decide, without loss of generality, to discard the last coordinates to define the projections. Since the $(N,d)$-Dirichlet DPP satisfies only~\eqref{hyp:sep_id_kern}, the choice of directions has an influence. For this process, Ripley's functions have been computed using a Monte-Carlo approach (based on $10^4$ replications): the coordinates to be removed are randomly chosen. The plots for the $(N,d)$-Dirichlet DPPs represent therefore the empirical mean of Ripley's functions. First and third quartiles are also represented by envelops to get an idea of the variability. The visual results show that for $\rho_\XX=500$, the $(N,d)$-Dirichlet DPP is the most repulsive among the three models. Moreover, the loss of repulsiveness when projecting turns out to be smaller for $(N,d)$-Dirichlet DPPs than for the two other DPP models. The envelops reported for the $(N,d)$-Dirichlet should be taken with attention. We could be tempted to conclude that the quite high variability observed for $d=6,10$, is too important to get practical interesting results. However, Section~\ref{sec:app} will discredit this argument. 

The $(N,d)$-Dirichlet DPP is the most repulsive in the situations considered here. However, it is worth mentioning that it may behave very badly according to the value of $N$. For example, we have observed that the less $N$ has factors the less repulsive the $(N,d)$-Dirichlet DPP. The values of these factors also affect the repulsiveness of the DPP. In particular, if $N$ is a high prime number, both situations are encountered which yields a disastrous model in terms of repulsion. 
Figures~\ref{fig:Ripley_6d}-\ref{fig:Ripley_100d} underline that the class of $L^1$-Exponential DPP is definitely less interesting than the class of Gaussian DPP. Given an~$\iota$, Ripley's function is closer to 1 and the convergence to 1 when $\iota$ decreases is faster for $L^1$-Exponential DPP. \rev{For this reason, the $L^1$-Exponential DPP is not considered in the next section.}

\begin{figure}
	\begin{changemargin}{-1.5cm}{0cm}
		\centering
		\includegraphics[scale=.225]{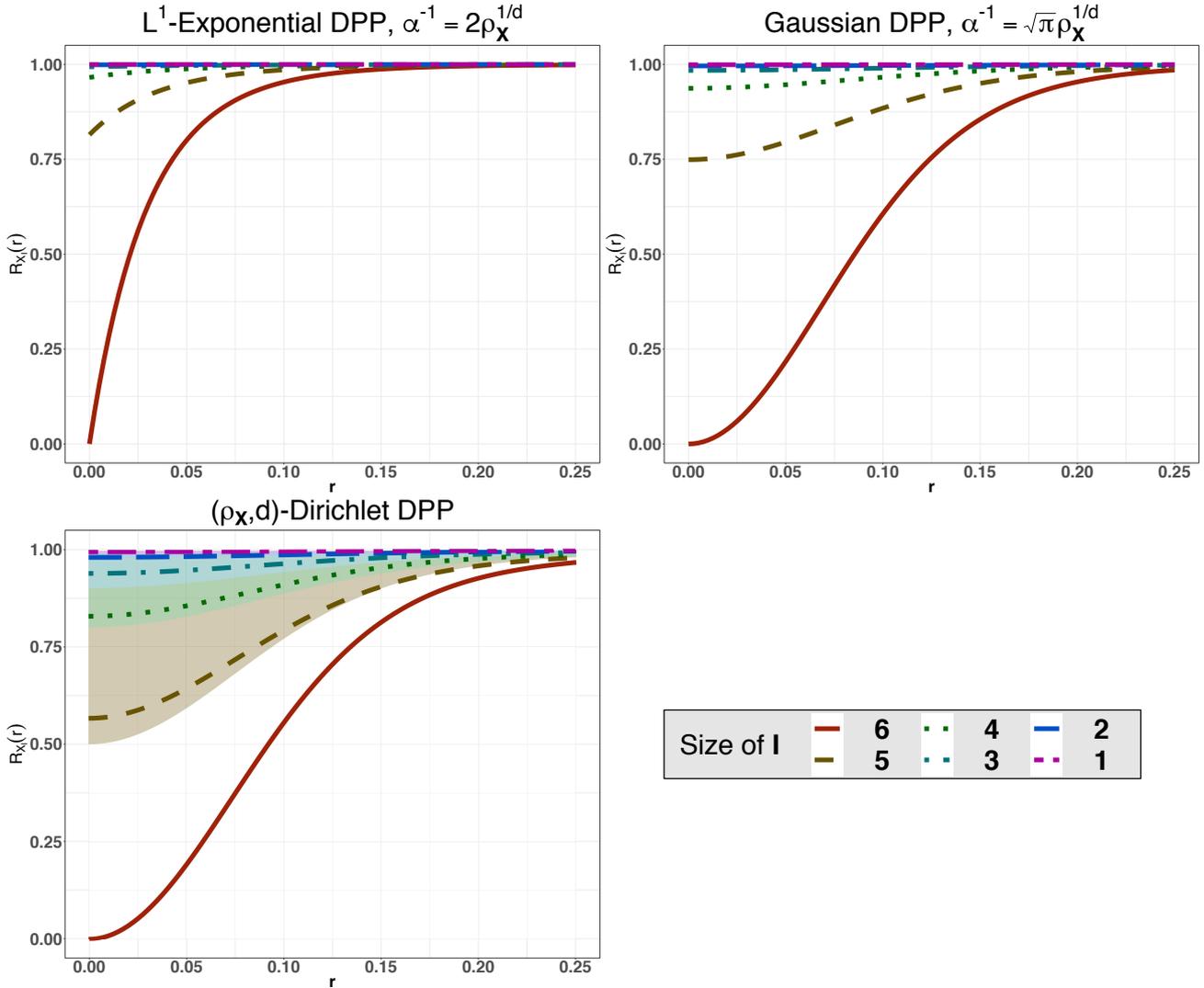}
		\end{changemargin}
	\caption{\rev{Ripley's functions (see~\eqref{eq:Ripley2}) of a $6$-dimensional DPP $\XX$ with intensity $\rho_\XX$ (solid lines) and its successive projections $\XX_I$ ($\abs{I}=d-1,\ldots$; dotted and dashed lines) for the $L^1$-Exponential DPP (top-left), Gaussian DPP (top right) and the $(N,6)$-Dirichlet DPP (bottom left). 
	For the Dirichlet case, coordinates to be removed are chosen randomly ($10^4$ replications): dotted and dashed lines represent empirical means while first and third quartiles are represented by envelops.}}
	\label{fig:Ripley_6d}
	\end{figure}
	
	\begin{figure}
	\begin{changemargin}{-1.5cm}{0cm}
		\centering
		\includegraphics[scale=.225]{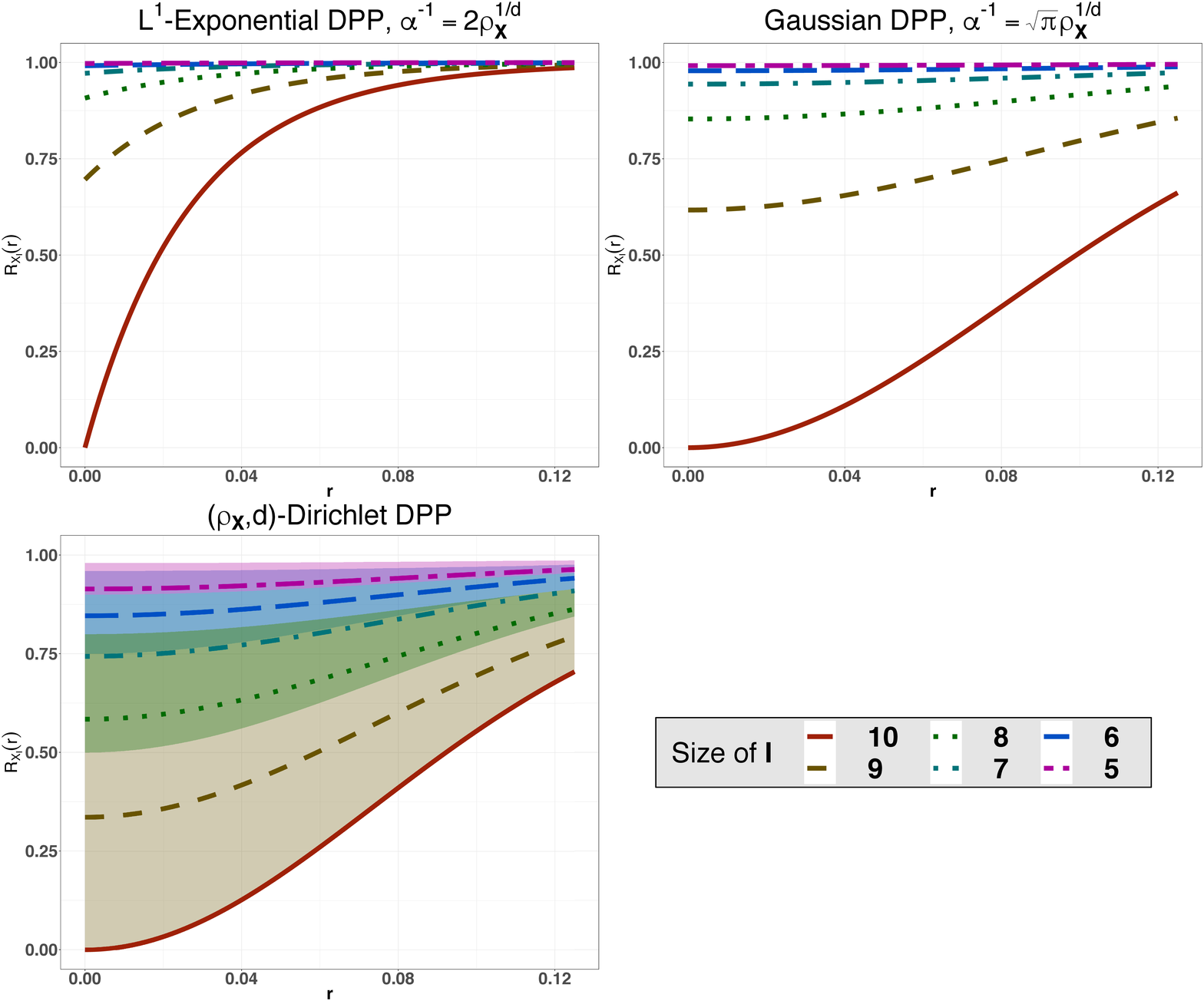}
		\end{changemargin}
	\caption{\rev{Ripley's functions (see~\eqref{eq:Ripley2}) of a $10$-dimensional DPP $\XX$ with intensity $\rho_\XX$ (solid lines) and its successive projections $\XX_I$ ($\abs{I}=d-1,\ldots$; dotted and dashed lines) for the $L^1$-Exponential DPP (top-left), Gaussian DPP (top right) and the $(N,10)$-Dirichlet DPP (bottom left). 
	For the Dirichlet case, coordinates to be removed are chosen randomly ($10^4$ replications): dotted and dashed lines represent empirical means while first and third quartiles are represented by envelops.}}
	\label{fig:Ripley_10d}
	\end{figure}
	
	\begin{figure}
	\begin{changemargin}{-1.5cm}{0cm}
		\centering
		\includegraphics[scale=.225]{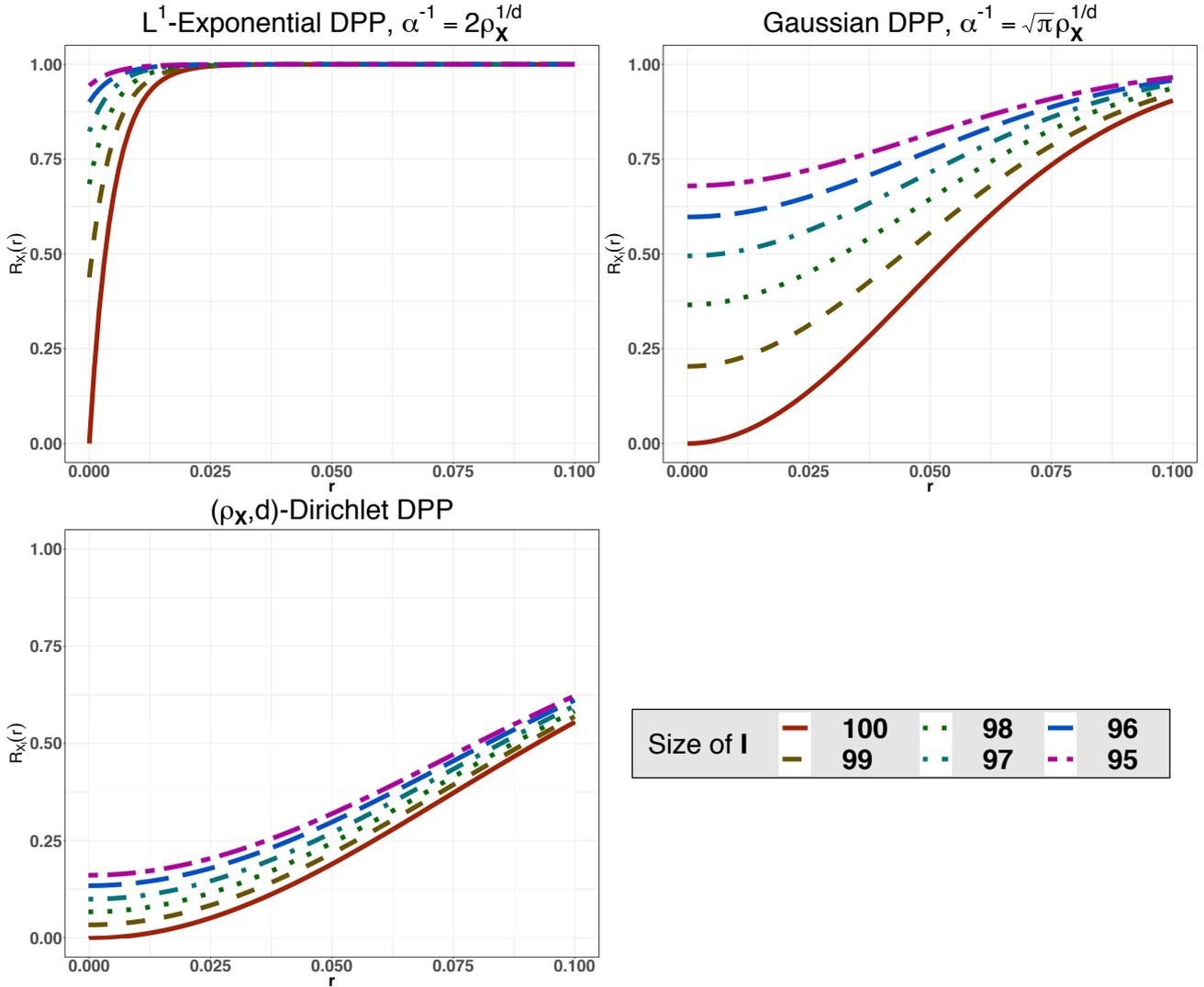}
		\end{changemargin}
	\caption{
	\rev{Ripley's functions (see~\eqref{eq:Ripley2}) of a $100$-dimensional DPP $\XX$ with intensity $\rho_\XX$ (solid lines) and its successive projections $\XX_I$ ($\abs{I}=d-1,\ldots$; dotted and dashed lines) for the $L^1$-Exponential DPP (top-left), Gaussian DPP (top right) and the $(N,100)$-Dirichlet DPP (bottom left). 
	For the Dirichlet case, coordinates to be removed are chosen randomly ($10^4$ replications): dotted and dashed lines represent empirical means while first and third quartiles are represented by envelops.}
	}
	\label{fig:Ripley_100d}
	\end{figure}

\section{Numerical illustrations} \label{sec:app}

In this section, \rev{we illustrate the interest of projected DPP models by simulation experiments}. For some  $d\ge 1$ and $I\subseteq \overline d$, the problem we consider is to estimate using a Monte-Carlo approach, an integral of the form
\[
 	\mu(f_I) = \int_{\rev{\I{0}{1}^d}} f_I(u) \de u
 \] 
 where $f_I:\rev{\I{0}{1}^{\iota}} \to \RR^+$ is a $\iota$-dimensional function. A standard way for achieving this task (which includes the uniform sampling design) is to define a point process, say $\mathbf Z_I$, on $\rev{\I{0}{1}^\iota}$ and estimate $\mu(f_I)$ using the unbiased estimator
 \begin{equation} \label{eq:MCM_est}
\widehat{\mu}_{\ZZZ_I}(f_I) = \rho_{\ZZZ_I}^{-1}\sum_{u\in\ZZZ_I} f_I(u).
\end{equation}
Given $I$ and $f_I$, this problem has been widely considered in the literature \cite[e.g.][]{RobertCasella04, DelyonPortier16}. In particular, an ad-hoc DPP on $\rev{\I{0}{1}^\iota}$, for which very interesting asymptotic results have been shown, has been proposed in \cite{BardenetHardy20}. In this section, we investigate another aspect. We consider the problem not only for one but various integrals, defined for different subsets $I \subseteq \overline d$ and based on a single realization of a point process defined on \rev{$\I{0}{1}^d$}. This problem, for which investigated models are definitely meaningful, mimics problems encountered in computer experiments where the spatial design is initially defined on $\RR^d$ but later used with a few coordinates  discarded \cite[e.g.][]{WoodsLewis16, Kleijnen17}.

To do this, we therefore consider a spatial point process $\XX$ (and in particular DPP models developed in the previous \rev{section}) and we estimate $\mu(f_I)$ by \eqref{eq:MCM_est} with $\ZZZ_I=\XX_I$ where $\XX_I$ is the projected point pattern of $\XX$ on $\rev{\I{0}{1}^\iota}$. The interest of our models lies in the following equation which evaluates $\mathrm{Var}(\widehat{\mu}_{\XX_I}(f_I))$. Using Campbell Theorem~\eqref{eq:rhok_def}
\begin{align} \label{eq:varI}
\mathrm{Var} \left( \widehat{\mu}_{\XX_I}(f_I) \right) &  = \rho_{\XX_I}^{-1} \int_{\rev{\I{0}{1}^\iota}} f_I(u)^2 \de u\\
& \quad+ \int_{\rev{\I{0}{1}^\iota}} \int_{\rev{\I{0}{1}^\iota}} (g_{\XX_I}(u,v)-1)f_I(u)f_I(v)\de u \de v.	 \nonumber
\end{align}
As soon as $g_{\XX_I}<1$, the variance is smaller than the first \rev{integral} which turns out to be the variance under the Poisson case. \rev{In this section, we intend to verify this property with models considered in this paper.}

In the following, we let $d=6$ and, \rev{following~\cite[Sec.~3]{BardenetHardy20},} we consider for any $I\subseteq \overline 6$ the ``bump'' test function
\begin{equation}\label{eq:fI}
f_I(u) = \exp\left(-\sum_{i\in I}\rev{ \frac{1}{1-4(u_i-1/2)^2}  }\right),\quad u\in \rev{\I{0}{1}^\iota}.
\end{equation}
Three type of models are investigated: a homogeneous Poisson point process (which serves as a reference), a Gaussian DPP,
and an $(N,6)$-Dirichlet DPP. Simulations of DPPs can be realized using \texttt{R} package \texttt{spatstat}. \rev{However, using this package leads to performance issues (in particular in terms of memory) when simulating DPPs with high intensity and/or high dimension. Therefore, we have implemented the simulation algorithms in C++ and have made them usable with R.}
 The codes are available on GitHub (\url{https://github.com/AdriMaz/rcdpp/}).

\begin{figure}
	\begin{changemargin}{-2.5cm}{0cm}
	\centering
	\includegraphics[scale=.19]{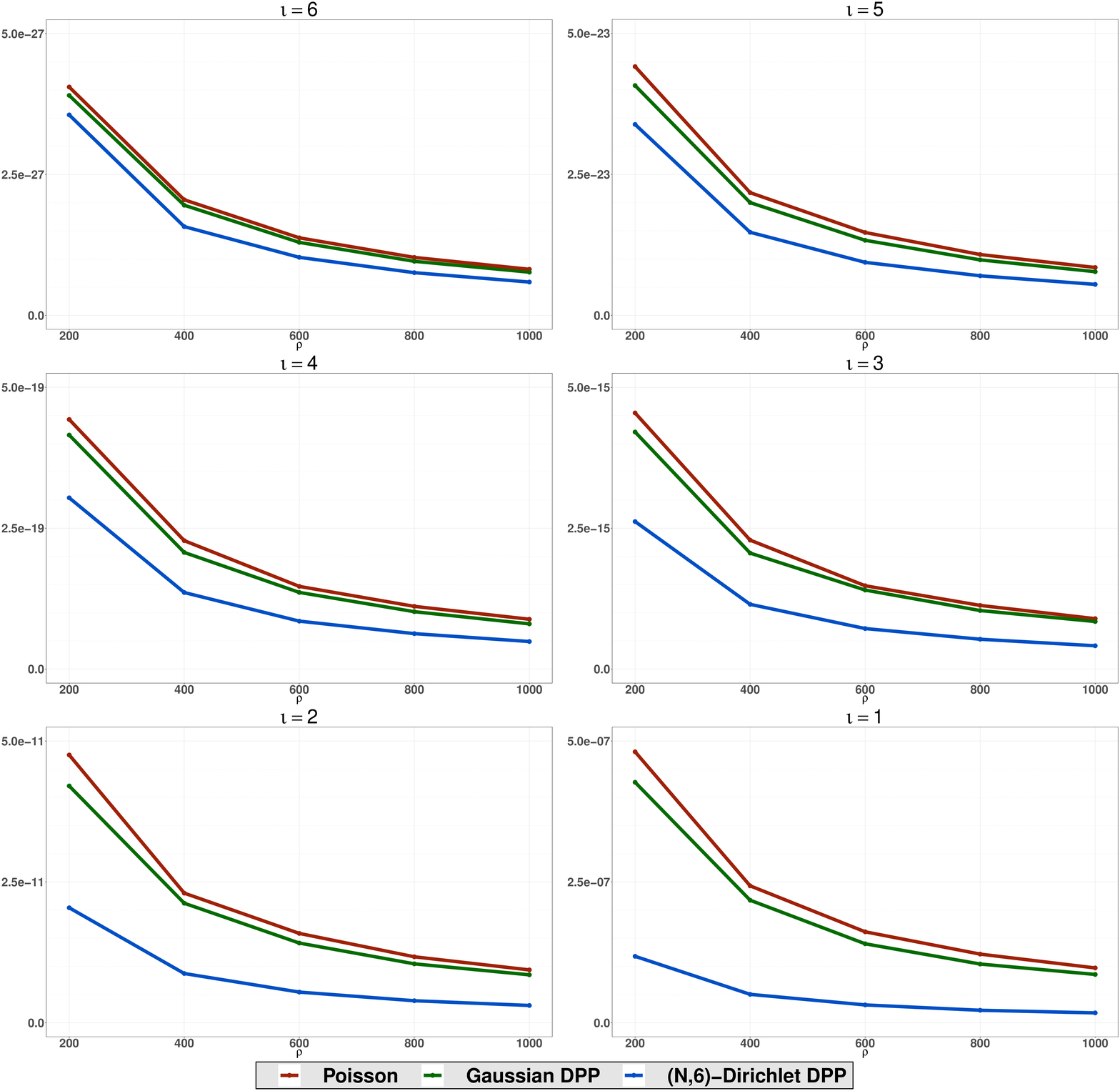}
		\end{changemargin}
		\caption{\rev{Empirical variances of Monte-Carlo integral estimates of the form~\eqref{eq:MCM_est} for the function \eqref{eq:fI} using Poisson process (red lines), Gaussian DPP (green lines) and Dirichlet DPP (blue lines) for $\iota=\abs{I}=6,\ldots, 1$, based on $10^4$ replications of a $6$-dimensional point processes with intensity \mbox{$\rho_\XX=200,400,600,800,1000$}. When $\iota<6$, coordinates to be removed are chosen randomly.}}
	\label{fig:MCM_resProj}
	\end{figure}
Figure~\ref{fig:MCM_resProj} reports empirical  variances of estimates of $\mu(f_I)$ based on \mbox{$m=10^4$} replications of each model, in terms of \rev{$\rho_\XX$ where $\rho_\XX=200,400,600,800,1000$}. We consider all possible projections, i.e. $\iota=6,5,4,3,2,1$. For the Poisson case, note that $\XX_I$ has the same distribution as a homogeneous Poisson point process (with the same intensity) defined on $B_I$.
For the Gaussian DPP, the parameter $\alpha$ is set to $\alpha^{-1}=\sqrt{\pi} \rho^{1/6}$. 
When $\iota<d$, the coordinates to be discarded are chosen randomly. This has no influence for the Poisson, Gaussian DPP
since these models satisfy Assumption~\eqref{hyp:sep_id_kern0} but is important for the $(N,6)$-Dirichlet DPP. 

Figure~\ref{fig:MCM_resProj} illustrates the interest of this research. It is clear that whatever the dimension of the function to integrate, i.e. whatever $\iota=6,\dots,1$, the empirical variance of Monte-Carlo estimates using one single realization of a spatial point process defined in dimension $d$, is always smaller than in the independent case. 
The \mbox{$(N,d)$-Dirichlet} model outperforms the \rev{Gaussian DPP} for any $I \subseteq \overline d$ \rev{as already observed from a theoretical point of view in the previous section}. The general result of this paper states that a projected DPP seems less and less repulsive after successive projections.  However, It is interesting to observe that this fact does not affect that much the properties of Monte-Carlo integration estimates.

\section*{Conclusion}

The objective of this paper is to explore properties of projections of a DPP $\XX$ with kernel $K$ and defined on a compact set $B$ of $\RR^d$. For any $I\subset \{1,\dots,d\}$, our general conclusion is that the projection $\XX_I$ remains repulsive when  kernel $K$ is separable, \rev{in the sense that $g_{\XX_I}<1$ uniformly for any $I\subset\{1,\dots,d\}$.} 
\rev{In particular if  kernel $K$ is a projection kernel,  $\XX_I$ falls in the class of $\alpha$-DPPs (with $\alpha=-1/n, n\in\mathbb{N}  	$). }}
We have proposed a few examples \rev{of such separable kernels}
and compared them using an original summary statistic based on a normalized version of the Ripley's function defined with the \rev{sup} norm. We have finally illustrated this paper for Monte-Carlo integration problems when the problem is to estimate integrals over a compact set $B_I$ of an $\iota$-dimensional function for any $1\le \iota \le d$, using the same quadrature points defined in $B$. 
\rev{To be fully relevant, comparisons with designs built from other classes of point processes (e.g. Gibbs, Mat\'ern, Multivariate OP Ensembles), more standard designs (e.g. Halton, Sobol, Quasi-Monte-Carlo), in which other test functions with different properties (e.g. less regular or non-compactly supported) would be considered, should be performed. We leave this for a future research.}

\section*{Acknowledgements} 

The authors would like to thank Fr\'ed\'eric Lavancier and Arnaud Poinas for fruitful discussions, \rev{and the associate editor and reviewers for valuable suggestions and comments}. The research of J.-F. Coeurjolly and A. Mazoyer is supported by the Natural Sciences and Engineering Research Council. P.-O. Amblard is partially funded by  Grenoble Data Institute  (ANR-15-IDEX-02) and LIA CNRS/Melbourne Univ Geodesic.


\printbibliography

\newpage


\appendix

\section{Proof of Lemma~\ref{lem:XI}} \label{app:proof_XI}
\begin{proof}
For any  non-negative measurable function $h_I:B_I^k\to \RR^+$, we have using Campbell Theorem \eqref{eq:rhok_def}
\begin{align*}
\int_{B_I^k} &h_I\left(x_I^{(1)},\ldots,x_I^{(k)}\right)  \rho_{\XX_I}^{(k)}\left(x_I^{(1)}, \ldots , x_I^{(k)}\right) \de x_I^{(1)}\ldots\de x_I^{(k)} \\
&= \EE\left[\sum_{x_I^{(1)}, \ldots , x_I^{(k)} \in \XX_I}^{\neq} h_I\left(x_I^{(1)}, \ldots , x_I^{(k)}\right)\right] \\
&= \EE\left[\sum_{x^{(1)}, \ldots , x^{(k)} \in \XX}^{\neq} (h_I \circ P_I)\left(x^{(1)}, \ldots , x^{(k)}\right)\right] \\
&= \int_{B_I^k} h_I \left(x_I^{(1)},\ldots,x_I^{(k)}\right)  
\bigg\{
\int_{(B_{I^c})^k} \rho_{\XX}^{(k)} \left(\left(x^{(1)},u^{(1)}\right),\ldots,\left(x^{(k)},u^{(k)}\right)\right) \\
& \qquad \qquad \de u^{(1)}\ldots\de u^{(k)}
\bigg\} \de x_I^{(1)}\ldots\de x_I^{(k)}
\end{align*}
whereby we deduce \eqref{eq:rhok_proj} by identification. 
\end{proof}

\section{Proof of Theorem~\ref{theo:proj_DPP_sep_kern}}  \label{app:proof_proj_DPP_sep_kern}

\begin{proof}
Let us write~\eqref{eq:rhok_proj_DPP1} under \eqref{hyp:sep_kern}.
	\begin{align} \label{eq:rhok_proj_DPP_sep_kern2}
		\rho_{\XX_I}^{(k)}\left(x^{(1)},\ldots,x^{(k)}\right)  & = \sum_{\sigma\in S_k} (-1)^{k-C(\sigma)} \int_{(B_{I^c})^k}\prod_{i=1}^k K\left((x,u)^{(i)},(x,u)^{(\sigma(i))}\right) \de u^{(1)}\ldots\de u^{(k)} \nonumber\\
		& = \sum_{\sigma\in S_k} (-1)^{k-C(\sigma)} \prod_{i=1}^k K_I\left(x_I^{(i)},x_I^{(\sigma(i))}\right)\nonumber \\
		& \phantom{\sum_{\sigma\in S_k} \qquad}\times \int_{(B_{I^c})^k} \prod_{i=1}^k K_{I^c}(u^{(i)},u^{(\sigma(i))}) \de u^{(1)}\ldots\de u^{(k)}.
	\end{align}
		For any~$\sigma\in S_k$ let us denote by~$\supp(\sigma)$ its support:
		\[
		\supp(\sigma) = \{i\in \overline{k}\mbox{  s.t.  } \sigma(i)\neq i\},
		\]
		by~$c(\sigma)$ the number of elements of~$\supp(\sigma)$, by~$\mathcal{S}(\sigma)$ the set of disjoint cycles of~$\sigma$ with non-empty support and by~$C(\sigma)$ the number of disjoint cycles of~$\sigma$ (including those with empty support).
		Consider the case where~$C(\sigma)=1$ (i.e.~$\sigma$ is a circular permutation of~$\overline{k}$). Then the integral part in  \eqref{eq:rhok_proj_DPP_sep_kern2} can be written as
\begingroup
\allowdisplaybreaks
\begin{align*}
			&\int_{(B_{I^c})^k}\prod_{i=1}^k  K_{I^c}\left(u^{(i)},u^{(\sigma(i))}\right)\de u^{(1)}\ldots \de u^{(k)} \nonumber\\[2ex]
			& = \int_{(B_{I^c})^k} K_{I^c}\left(u^{(1)},u^{(\sigma(1))}\right)\ldots K_{I^c}\left(u^{(\sigma(1))},u^{(\sigma^2(1))}\right)\ldots \\
			& \qquad \dots K_{I^c}\left(u^{(k)},u^{(\sigma(k))}\right) \de u^{(1)}\ldots\de u^{(\sigma(1))}\ldots \de u^{(k)} \nonumber\\[2ex]
			& = \int_{(B_{I^c})^{k-1}}  K_{I^c}^{(2)}\left(u^{(1)},u^{(\sigma^2(1))}\right)\ldots\nonumber\\
			&\phantom{\int_{(B_{I^c})^k} }\ldots K_{I^c}\left(u^{(\sigma(1)-1)},u^{(\sigma(\sigma(1)-1))}\right)K_{I^c}\left(u^{(\sigma(1)+1)},u^{(\sigma(\sigma(1)+1))}\right)\ldots\nonumber\\
			&\phantom{\int_{(B_{I^c})^k} }\ldots  K_{I^c}\left(u^{(k)},u^{(\sigma(k))}\right)\de u^{(1)}\ldots\de u^{(\sigma(1)-1)}\de u^{(\sigma(1)+1)}\ldots \de u^{(k)} \nonumber\\
			&\vdots \nonumber\\
			& = \int_{(B_{I^c})^2}  K_{I^c}^{(k-1)}\left(u^{(1)},u^{(\sigma^{k-1}(1))}\right) K_{I^c}\left(u^{(\sigma^{k-1}(1))},u^{(\sigma^k(1))}\right) \de u^{(1)} \de u^{(\sigma^{k-1}(1))}  \nonumber\\[2ex]
			& = \tr_{B_{I^c}} \left(K_{I^c}^{(k)}\right).
		\end{align*}
\endgroup
		Assume now that $C(\sigma)>1$. Then $\sigma$ can be written as \begin{equation} \label{eq:dec_perm}
			\sigma = \left(\bigodot_{\varepsilon\in \mathcal{S}(\sigma)}\varepsilon\right)\odot \mathbf{i}_k(\sigma),
		\end{equation}
		where~$\mathbf{i}_k(\sigma)$ is the identity on~$\overline{k}\setminus\supp(\sigma)$, and~$\odot$ denotes the permutation product.
		Observe that \eqref{eq:dec_perm} implies 
		\begin{equation}
			C(\sigma) = \#S(\sigma)+k-c(\sigma).
		\end{equation}
		If~$1\in\supp(\sigma)$, there is only one permutation~$\varsigma\in\mathcal{S}(\sigma)$ such that~$1\in\supp(\varsigma)$. Therefore:
		\begin{align*}
			&\int_{(B_{I^c})^k} \prod_{i=1}^k  K_{I^c}\left(u^{(i)},u^{(\sigma(i))}\right)\de u^{(1)}\ldots \de u^{(k)} \\[2ex]
			&=\tr_{B_{I^c}}\left(K^{(c(\iota))}\right)\int_{(B_{I^c})^{k-c(\varsigma)}}\prod_{i\in \overline{k}\setminus \supp(\varsigma)}K_{I^c}\left(u^{(i)},u^{(\sigma(i))}\right)\de u^{(1)}\ldots \de u^{(k)}
		\end{align*}
		Denote by~$\alpha$ the minimum of~$\overline{k}\setminus\supp(\varsigma)$. As above there is only one permutation $\varsigma'\in\mathcal{S}(\sigma)$ such that~$\alpha\in\supp(\varsigma')$. Then:
\begingroup
\allowdisplaybreaks
		\begin{align*}
			&\int_{(B_{I^c})^k} \prod_{i=1}^k  K_{I^c}\left(u^{(i)},u^{(\sigma(i))}\right)\de u^{(1)}\ldots \de u^{(k)} \\[2ex]
			&=\tr_{B_{I^c}}\left(K_{I^c}^{(c(\varsigma))}\right)\int_{(B_{I^c})^{k-c(\varsigma)}}\prod_{i\in \overline{k}\setminus \supp(\varsigma)}K_j\left(u^{(i)},u^{(\sigma(i))}\right)\de u^{(1)}\ldots \de u^{(k)}\\[2ex]
			&=\tr_{B_{I^c}}\left(K_{I^c}^{(c(\varsigma))}\right)\tr_{B_j}\left(K_{I^c}^{(c(\varsigma'))}\right)\\
			& \;\times\int_{(B_{I^c})^{k-(c(\varsigma)+c(\varsigma')}}\prod_{i\in \overline{k}\setminus \left(\supp(\varsigma)\cup\supp(\varsigma')\right)}K_{I^c}\left(u^{(i)},u^{(\sigma(i))}\right)\de u^{(1)}\ldots \de u^{(k)}.
		\end{align*}
\endgroup
		Therefore, one gets by induction:
		\begin{align}\label{eq:int_Kperp_gen}
			&\int_{(B_{I^c})^k} \prod_{i=1}^k  K_{I^c}\left(u^{(i)},u^{(\sigma(i))}\right)\de u^{(1)}\ldots \de u^{(k)}\nonumber \\[2ex]
			& = \left[\prod_{\varepsilon\in\mathcal{S}(\sigma)}\tr_{B_{I^c}}\left(K_{I^c}^{(c(\varepsilon))}\right)\right]\int_{(B_{I^c})^{k-c(\sigma)}} \prod_{i\in\overline{k}\setminus\supp(\sigma)}K_{I^c}\left(u^{(i)},u^{(\sigma(i))}\right)\de u^{(1)}\ldots\de u^{(k)}.
			\end{align}

		Plugging~\eqref{eq:int_Kperp_gen} into \eqref{eq:rhok_proj_DPP_sep_kern2} leads to \eqref{eq:rhok_proj_DPP_sep_kern_res}.

\end{proof}

\section{Laplace functionals} \label{sec:laplace}

\begin{lemma}\label{lem:laplace}
 Let $I\subset\overline{d}$ and let $\XX$ be a spatial point process defined on a compact set $B\subset \RR^d$. Then, for any Borel function~$h_I:B_I\rightarrow\RR^+$
\begin{equation}
	\label{eq:Laplce_proj}
\LL_{\XX_I}(h_I) = \LL_\XX(h_I \circ P_I).
\end{equation} 
\end{lemma}

\begin{proof}
Equation~\eqref{eq:Laplce_proj} follows  arguments similar to the ones used in the proof of Lemma~\ref{lem:XI}.
\[
\LL_{\XX_I}(h_I)  = \EE\left[\prod_{y\in \XX_I} \e{-h_I(y)}\right] 
= \EE\left[\prod_{x\in\XX} \e{-h_I(x_I)}\right]
= \LL_{\XX}(h_I\circ P_I)	
\]
\end{proof}

\begin{theorem} \label{thm:laplace}
Let $I\subseteq\overline{d}$ and $\XX\sim \DPP_B(K)$ such that $K$ satisfies \eqref{hyp:sep_kern}. The Laplace functional of the projected point process $\XX_I$ is given for any Borel function $h_I:B_I\rightarrow \RR^+$ by:
\begin{align}
\LL_{\XX_I}(h_I) 
&= \prod_{l\in \NN_{I^c}} \exp\left\{-\sum_{k\ge1}\frac{\tr_{B_I}\left({K}_{\lambda_l^{(I^c)}K_I,h_I}^{(k)}\right)}{k} \right\} \label{eq:lap_proj_DPP_sep_kern_res1} \\[2ex]
&= \exp\left\{-\sum_{k\ge1}\frac{\tr_{B_{I^c}}\left(K_{I^c}^{(k)}\right) \tr_{B_I}\left({K}_{I,h_I}^{(k)}\right)}{k} \right\} \label{eq:lap_proj_DPP_sep_kern_res2}
\end{align}
where  $K_{I,h_I}:B_I\times B_I\rightarrow\CC$ is the kernel defined by
\[
K_{I,h_I}(x,y) = \sqrt{1-\e{-h_I(x)}} K_I(x,y) \sqrt{1-\e{-h_I(y)}}.
\]
\end{theorem}

\begin{proof}
The proof is straightforward and follows from Lemma~\ref{lem:laplace}, assumption~\eqref{hyp:sep_kern} and the definition of $K_{I,h_I}$.
\end{proof}

\end{document}